\documentclass[12pt]{amsart}

\usepackage{graphicx}
\usepackage{epsfig}
\usepackage{fancyhdr}
\usepackage{amsmath,amsfonts,amssymb,amsthm}
\usepackage{framed}

\newtheorem{defn}{Definition}[section]

\newtheorem{cor}[defn]{Corollary}
\newtheorem{thm}[defn]{Theorem}
\newtheorem{prop}[defn]{Proposition}
\newtheorem{lemma}[defn]{Lemma}

\newcommand{\be}{\begin{equation}}
\newcommand{\ee}{\end{equation}}
\newcommand{\bea}{\begin{eqnarray}}
\newcommand{\eea}{\end{eqnarray}}
\newcommand{\beas}{\begin{eqnarray*}}

\newcommand{\eeas}{\end{eqnarray*}}

\newcommand{\bp}{\begin{proof}}
\newcommand{\ep}{\end{proof}}

\newcommand{\reftit}{\textit}    
\newcommand{\refis}{\textbf}     

\begin{document}

\title[Cross-Multiplicative Coalescent]{Cross-Multiplicative Coalescent Processes and Applications}

\author{Yevgeniy Kovchegov}
\address{Department of Mathematics, Oregon State University, Corvallis, OR  97331, USA}
\email{kovchegy@math.oregonstate.edu}
\thanks{This work was supported by the FAPESP award 2016/19286-0 and by the NSF award DMS-1412557.}

\author{Peter T. Otto}
\address{Department of Mathematics, Willamette University, Salem, OR 97302, USA}
\email{potto@willamette.edu}

\author{Anatoly Yambartsev}
\address{Department of Statistics, Institute of Mathematics and Statistics, University of S\~ao Paulo, Rua do Mat\~ao 1010,
CEP 05508--090, S\~ao Paulo SP, Brazil}
\email{yambar@ime.usp.br} 


\begin{abstract}
We introduce and analyze a novel type of coalescent processes called {\it cross-multiplicative coalescent} 
that models a system with two types of particles, $A$ and $B$. The bonds are formed only between the pairs of particles of opposite types 
with the same rate for each bond, 
producing connected components made of particles of both types.
We analyze and solve the Smoluchowski coagulation system of equations obtained as a hydrodynamic limit of  the corresponding Marcus-Lushnikov process. 
We establish that the cross-multiplicative kernel is a gelling kernel, and find the gelation time. As an application, we derive the limiting mean length of a minimal spanning tree on a complete bipartite graph $K_{\alpha[n], \beta[n]}$ with partitions of sizes $\alpha[n]=\alpha n +o(\sqrt{n})$ and $\beta[n]=\beta n +o(\sqrt{n})$ and independent edge weights, distributed uniformly over $[0, 1]$. 

\end{abstract}

\date{\today}

\maketitle

\tableofcontents

\section{Introduction}\label{sec:intro}

The coalescence dynamics of clusters with multidimensional weight (mass) vectors was originally considered in Krapivsky and Ben-Naim \cite{KBN98} and Vigil and Ziff \cite{VZ98} in the context of aggregation kinetics with applications to aerosol dynamics and copolymerization kinetics.
In this paper, we consider a coalescent process whose clusters have vector-valued weights in $\mathbb{R}_+^2$.
The coalescent process begins with $\alpha[n]=\alpha n +o(\sqrt{n})$ singletons of weight $\left[\!\!\begin{array}{c}1 \\0\end{array}\!\!\right]$ and $\beta[n]=\beta n +o(\sqrt{n})$ singletons of weight $\left[\!\!\begin{array}{c}0 \\1\end{array}\!\!\right]$. 
This continuous time Markov process evolves as follows.
Each pair of clusters with respective weight vectors ${\bf i}=\left[\!\!\begin{array}{c}i_1 \\i_2\end{array}\!\!\right]$ and ${\bf j}=\left[\!\!\begin{array}{c}j_1 \\j_2\end{array}\!\!\right]$
has the rate $K({\bf i},{\bf j})/n$ for coalescing into a cluster of weight ${\bf i}+{\bf j}$, where 
$$K({\bf i},{\bf j})=i_1j_2+i_2j_1$$
is the {\it cross-multiplicative coalescent kernel} governing the coalescent process. 
Such process will be called the {\it cross-multiplicative coalescent process}.

\medskip
\noindent
As a physical model, one may consider a system with two types of particles, $A$ and $B$. 
The process begins with $\alpha[n]$ particles of type $A$ and $\beta[n]$ particles of type $B$.
Each particle interacts only with the particles of opposite type, with which it may form a bond.
The bonds are formed independently, each with rate $1/n$.
Thus, the bonds may be formed only between the pairs of particles of opposite types, producing connected components (clusters). 
In these clusters, each pair of neighbor vertices will be of opposite type.
The model can be interpreted as a bond percolation model on a complete bipartite graph $K_{\alpha[n], \beta[n]}$ with 
the probability of an edge being open $p=1-e^{-t/n}$ increasing from zero to one as the time $t$ increases from zero to infinity.
See Subsect.~\ref{sec:ERonKnn}.

\medskip
\noindent
Let $\zeta_{i_1,i_2}^{[n]}(t)$ denote the number of the components of weight $\left[\!\!\begin{array}{c}i_1 \\i_2\end{array}\!\!\right]$ at time $t$.
The hydrodynamic limit
$$\lim\limits_{n \to \infty} \sup\limits_{s \in [0,T]} \left|n^{-1}\zeta_{i_1,i_2}^{[n]}(s)-\zeta_{i_1,i_2}(s)\right|=0 \qquad \text{ a.s.}$$
for an arbitrary $T>0$ is established in Sect.~\ref{sec:hydcross} via the weak convergence results of Kurtz \cite{EK,Kurtz81} 
for density dependent population processes.
The limiting functions $\zeta_{i_1,i_2}(t)$ indexed by $\mathbb{Z}_+^2 \setminus \{(0,0)\}$ are expressed as the solutions of the following {\it modified Smoluchowski coagulation system} of differential equations
\begin{equation}\label{eqn:Knnt_intro}
{d \over dt}\zeta_{i_1,i_2}(t)=-(\beta i_1+ \alpha i_2)\zeta_{i_1,i_2}(t)+ {1 \over 2}\sum\limits_{\substack{\ell_1, k_1:~ \ell_1+k_1=i_1,\\ \ell_2, k_2:~ \ell_2+k_2=i_2}} (\ell_1 k_2+\ell_2 k_1)\zeta_{\ell_1,\ell_2}(t)\zeta_{k_1,k_2}(t)
\end{equation}
with the initial conditions  $\zeta_{i_1,i_2}(0)=\alpha \delta_{1,i_1}\delta_{0,i_2}+\beta \delta_{0,i_1}\delta_{1,i_2}$.
The above system has a unique solution as established in the following theorem.

\medskip
\noindent
{\bf Sect.~\ref{sec:knn}, Theorem \ref{thm:solutionKnnt}.}
{\it The modified Smoluchowski coagulation system of equations \eqref{eqn:Knnt_intro} with the initial conditions $~\zeta_{i_1,i_2}(0)=\alpha \delta_{1,i_1}\delta_{0,i_2}+\beta \delta_{0,i_1}\delta_{1,i_2}$ has the unique solution
$$\zeta_{i_1,i_2}(t)={i_1^{i_2-1}i_2^{i_1-1}\alpha^{i_1}\beta^{i_2} \over i_1!i_2!}e^{-(\beta i_1+ \alpha i_2)t}t^{i_1+i_2-1}.$$
}

\medskip
\noindent
This solution of the coagulation system enables us to establish gelation in the cross-multiplicative coalescent process.

\medskip
\noindent
{\bf Sect.~\ref{sec:knn}, Corollary \ref{cor:GelationTimeKanbn}.}
{\it The cross-multiplicative kernel is a gelling kernel, with the gelation time given by
$$T_{gel}={1 \over \sqrt{\alpha \beta}}.$$ 
}

\medskip
\noindent
We already mentioned the connection of cross-multiplicative coalescece to aggregation kinetics \cite{KBN98,VZ98}.
Besides this, the study of the cross-multiplicative coalescent process is justified by its relation to the Erd\H{o}s-R\'enyi process on $K_{\alpha[n], \beta[n]}$.
This relation will be used in the applications presented in Sect.~\ref{sec:applications} of this paper.

\subsection{Applications in minimal spanning trees}\label{sec:intro-apps}
As an application, we attempt to extend the connection between coalescent processes and random graph processes, e.g. Erd\H{o}s-R\'enyi random graph evolution as described in Sect.~\ref{sec:applications}. In particular, deriving a formula for the limiting length of the minimal spanning tree in a random graph process in terms of the solutions of the Smoluchowski coagulation equations for the corresponding coalescent process.

It is well known that in many cases the cluster dynamics of a random graph process can be replicated with a coalescent process. For example, the Erd\H{o}s-R\'enyi random graph process on $K_n$ can be tied to the $n$-particle multiplicative coalescent (see Aldous \cite{Aldous}). The connection lies in that the probability of two components merging, at a given time, depends only on the number of edges that connect those two components (rather than other structural properties). There are more elaborate examples. 

The cluster dynamics of a coalescent process (without merger history) is traced by an auxiliary process called the Marcus-Lushnikov process. The merger dynamics of such coalescent processes corresponds to a greedy algorithm for finding the minimal spanning tree in the respective random graph process. This observation allows us to express the limiting mean length of a minimal spanning tree in terms of the solutions of the Smoluchowski coagulation equations that represent the hydrodynamic limit of the Marcus-Lushnikov process corresponding to the random graph process. 

As a particular application of the proposed general approach we find the asymptotic limit for the mean length of a minimal spanning tree for the complete bipartite graph with partitions of sizes $\alpha[n]=\alpha n +o(\sqrt{n})$ and $\beta[n]=\beta n +o(\sqrt{n})$. See Sect.~\ref{sec:applications}. There, the probability of two components merging at a given time depends only on the number of edges that connect those two components. If connected component $C_i$ and $C_j$ have partition sizes $(i_1,i_2)$ and $(j_1,j_2)$ respectively, then there are $i_1j_2+i_2j_1$ edges which, when opened, would connect $C_i$ and $C_j$. 

\medskip
\noindent
{\bf Sect.~\ref{sec:applications}, Theorem \ref{thm:mainKnn}.}
{\it Let $\alpha, \beta >0$ and $L_n=L_n(\alpha,\beta)$ be the length of a minimal spanning tree on a complete bipartite graph $K_{\alpha[n], \beta[n]}$ with partitions of sizes $$\alpha[n]=\alpha n +o(\sqrt{n}) ~\text{ and }~\beta[n]=\beta n +o(\sqrt{n})$$
and independent uniform edge weights over $[0, 1]$. Then
\begin{equation*}
\lim\limits_{n \rightarrow \infty} E[L_n]=\sum\limits_{i_1,i_2}^\infty \int\limits_0^\infty \zeta_{i_1,i_2}(t)d(t),
\end{equation*}
where $\zeta_{i_1,i_2}(t)$ indexed by $\mathbb{Z}_+^2 \setminus \{(0,0)\}$ is the solution of the following system of equations 
\begin{equation*}
{d \over dt}\zeta_{i_1,i_2}(t)=-(\beta i_1+ \alpha i_2)\zeta_{i_1,i_2}(t)+ {1 \over 2}\sum\limits_{\substack{\ell_1, k_1:~ \ell_1+k_1=i_1,\\ \ell_2, k_2:~ \ell_2+k_2=i_2}} (\ell_1 k_2+\ell_2 k_1)\zeta_{\ell_1,\ell_2}(t)\zeta_{k_1,k_2}(t)
\end{equation*}
with the initial conditions  $\zeta_{i_1,i_2}(0)=\alpha \delta_{1,i_1}\delta_{0,i_2}+\beta \delta_{0,i_1}\delta_{1,i_2}$.
}

\medskip
\noindent
Recall that the above system of equations is the modified Smoluchowski coagulation system (\ref{eqn:Knnt}) of the cross-multiplicative coalescent process with 
the cross-multiplicative kernel introduced in \eqref{biK}. 

\medskip
\noindent
In Theorem \ref{thm:solutionKnnt}, the system of equations in Theorem \ref{thm:mainKnn} is solved. This yields the main result of Sect.~\ref{sec:applications}.

\medskip
\noindent
{\bf Sect.~\ref{sec:applications}, Theorem \ref{thm:main}.}
{\it Let $\alpha, \beta >0$, $\gamma=\alpha/\beta$, and $L_n=L_n(\alpha,\beta)$ be the length of a minimal spanning tree on a complete bipartite graph $K_{\alpha[n], \beta[n]}$ with partitions of sizes $$\alpha[n]=\alpha n +o(\sqrt{n}) ~\text{ and }~\beta[n]=\beta n +o(\sqrt{n})$$
and independent uniform edge weights over $[0, 1]$. Then the limiting mean length of the minimal spanning tree is
$$\lim\limits_{n \rightarrow \infty} E[L_n] =\gamma+{1 \over \gamma}+\sum\limits_{i_1 \geq 1;~i_2 \geq 1}{(i_1+i_2-1)!\over i_1!i_2!}{\gamma^{i_1}i_1^{i_2-1}i_2^{i_1-1}\over (i_1+\gamma i_2)^{i_1+i_2}}.$$
}

\medskip
\noindent
The above result is novel for $\alpha \not= \beta$, when the complete bipartite graph $K_{\alpha[n], \beta[n]}$ is an irregular graph.
For $\alpha=\beta$, Theorem \ref{thm:main} recovers the result of Frieze and McDiarmid \cite{FMcD}, as stated in the following corollary that we also prove in Section  \ref{sec:knn}.
\begin{cor}[Sect.~\ref{sec:applications}, Corollary \ref{cor:FMcD}]
If $\alpha=\beta$, then
$$\lim\limits_{n \rightarrow \infty} E[L_n]=2\zeta(3).$$
\end{cor}

\bigskip
\noindent
The paper is organized as follows.
Sect.~\ref{sec:background} provides the background on coalescent processes, multiplicative coalescent, and gelation.
In Sect.~\ref{sec:knn} the cross-multiplicative coalescent process and the corresponding Marcus-Lushnikov process are analyzed.
Sect.~\ref{sec:applications} gives applications of multiplicative and cross-multiplicative coalescent processes in minimal spanning trees.
Finally, in Section \ref{sec:weakconv}, the weak convergence results of Kurtz \cite{EK,Kurtz81} are applied to Marcus-Lushnikov processes with multiplicative and cross-multiplicative kernels. The paper concludes with a discussion in Section \ref{sec:dscussion}.

\section{Background on coalescent processes and gelation}\label{sec:background}
A general finite coalescent process begins with $n$ singletons (clusters of mass one). The cluster formation is governed by a symmetric collision rate 
kernel $K(i,j)=K(j,i)>0$. Specifically, a pair of clusters with masses (weights) $i$ and $j$ coalesces at the rate $K(i,j)/n$, independently of the other pairs, to form a new cluster of mass $i+j$. The process continues until there is a single cluster of mass $n$. See  \cite{Pitman,Aldous,Bertoin,Berestycki,Evans} and references therein.

\medskip
\noindent
Formally, for a given $n$ consider the space $\mathcal{P}_{[n]}$ of partitions of $[n]=\{1,2,\hdots,n\}$. 
Let $\Pi^{(n)}_0$ be the initial partition in singletons, and $\Pi^{(n)}_t ~~(t \geq 0)$ be a strong Markov process such that 
$\Pi^{(n)}_t$ transitions from partition $\pi \in \mathcal{P}_{[n]}$ to $\pi' \in \mathcal{P}_{[n]}$ with rate $K(i,j)/n$ provided that partition 
$\pi'$ is obtained from partition $\pi$ by merging two clusters of $\pi$ of weights $i$ and $j$.
If $K(i,j) \equiv 1$ for all positive integer masses $i$ and $j$, the process $\Pi^{(n)}_t$ is known as Kingman's $n$-coalescent process. If $K(i,j)=i+j$ the process is called $n$-particle additive coalescent. Finally, if $K(i,j)=ij$ the process is called $n$-particle multiplicative coalescent. The so called Marcus-Lushnikov process 
\begin{equation}\label{eqn:MLprocess}
{\bf ML}_n(t)=\Big(\zeta_1^{[n]}(t),\zeta_2^{[n]}(t),\hdots, \zeta_n^{[n]}(t), 0,0,\hdots \Big)
\end{equation}
is an auxiliary process to the corresponding coalescent process that keeps track of the numbers of clusters in each weight category. Here we let $\zeta_k^{[n]}(t)$ denote the number of clusters of mass $k$ in a multiplicative coalescent process of $n$ particles at time $t\geq 0$. See \cite{Marcus} and \cite{Lushnikov} for the original papers by Marcus and Lushnikov. The latter work considered the gelation phenomenon emerging in some of the Marcus-Lushnikov processes. The Marcus-Lushnikov process does not differentiate between the clusters of the same weight, and therefore does not keep track of the merger history of the $n$-particle coalescent process.

The deterministic dynamics of the limiting fractions  $\zeta_k(t)=\lim\limits_{n \rightarrow \infty} {\zeta_k^{[n]}(t) \over n}$ of clusters of size $k$ is described by the Smoluchowski system of coagulation equations  \cite{Smol16} or by its modified version, the Flory equations. See \cite{Flory,Jeon98,Norris99,FG04,FL09}.
The general system of Smoluchowski coagulation equations with a positive symmetric kernel $K(i,j)$ is the following mean-field approximation of 
coalescent dynamics
\begin{equation}\label{eqn:GeneralSmoluchowski}
{d \over dt}\zeta_j=-\zeta_j\sum_{i=1}^\infty K(i,j) \zeta_i +{1 \over 2}\sum_{i=1}^{j-1}K(i,j-i)\zeta_i\zeta_{j-i} \qquad (k=1,2,\hdots).
\end{equation}

\medskip
\noindent
One of the important questions in the theory of Smoluchowski equations is whether the {\it conservation of mass} property
\begin{equation}\label{eqn:conservation_mass}
\sum_{j=1}^\infty j\zeta_j(t)=\sum_{j=1}^\infty j\zeta_j(0)
\end{equation}
holds for all $t\geq 0$, or if there exists a time $T_{gel}<\infty$ after which the total mass $\sum_{j=1}^\infty j\zeta_j$ begins to dissipate.

\subsection{Gelation}\label{sec:gel}
The phenomenon of loosing total mass after a certain finite time $T_{gel}$ is called {\it gelation}. 
Time $T_{gel}>0$, if finite as in the multiplicative case, is called the {\it gelation time}. The kernel function $K(\cdot,\cdot)$ for which such $T_{gel}<\infty$ is called the {\it gelling kernel}. 
Informally, the gelation time corresponds to formation of a giant cluster called the {\it gel}.
The gelation phenomenon was studied extensively in the coagulations equations literature. 
See \cite{Aldous98, Aldous, vDE86, Jeon99, Lushnikov} and references therein.
Here, we would like to summarize some (but not all) of the concepts and results concerning the gelation phenomenon. 

\medskip
\noindent
Consider a general system \eqref{eqn:GeneralSmoluchowski} of Smoluchowski coagulation equations with a positive symmetric kernel $K(i,j)$,
and given initial conditions $\zeta_j(0)$.
Then, following \cite{McLeod62}, we use the Smoluchowski equations \eqref{eqn:GeneralSmoluchowski} to obtain 
\begin{align*}
{d \over dt}\sum_{j=1}^\infty j\zeta_j = \sum_{j=1}^\infty j{d \over dt}\zeta_j 
& =  -\sum_{i,j=1}^\infty j K(i,j) \zeta_j \zeta_i +{1 \over 2}\sum_{j=1}^\infty \sum_{i=1}^{j-1}\big(i+(j-i)\big)K(i,j-i)\zeta_i\zeta_{j-i}\\
& = -\sum_{i,j=1}^\infty j K(i,j) \zeta_j \zeta_i +{1 \over 2}\sum_{i,j=1}^\infty \big(i+j\big)K(i,j)\zeta_i\zeta_j ~ = 0 
\end{align*}
provided convergence of $\,\sum\limits_{i,j=1}^\infty j K(i,j) \zeta_j \zeta_i$. 
Thus, letting the gelation time be defined via the following critical transition,
\begin{equation}\label{eqn:gelationGeneral01}
T_{gel}:=\inf\Big\{t>0~:~\sum\limits_{i,j=1}^\infty j K(i,j) \zeta_j(t) \zeta_i(t)=\infty \Big\},
\end{equation}
we have ${d \over dt}\sum_{j=1}^\infty j\zeta_j =0$ for $t \in [0,T_{gel})$, which in turn implies
\eqref{eqn:conservation_mass} for $t \in [0,T_{gel})$.

Suppose the hydrodynamic limit $\,\lim\limits_{n \rightarrow \infty} {\zeta_k^{[n]}(t) \over n}=\zeta_k(t)$ is established for the Marcus-Lushnikov process 
with the given kernel $K(i,j)$, where $\zeta_k(t)$ is the solution of a coagulation system of equations.
See \cite{Norris99,FG04,FL09}. Then the definition 
of gelation time in formula \eqref{eqn:gelationGeneral01} is replaced with 
 \begin{equation}\label{eqn:gelationGeneral02}
 T_{gel}:=\inf\Big\{t>0~:~\sum_{j=1}^\infty j\zeta_j(t)<\sum_{j=1}^\infty j\zeta_j(0) \Big\}.
 \end{equation}
While \eqref{eqn:gelationGeneral01} relies on the explosion of higher moments (often, the second moment $\sum_j j^2\zeta_j$) and \eqref{eqn:gelationGeneral02} concerns the behavior of the first moment, the two definitions of gelation are usually equivalent. 

Weak convergence of the Marcus-Lushnikov processes to either a Smoluchowski system or a modified Smoluchowski (Flory) system was explored 
in Jeon \cite{Jeon98}, Norris \cite{Norris99}, Fournier and Giet \cite{FG04}, and in Fournier and Lauren\c{c}ot \cite{FL09}.
Specifically, it was shown in Fournier and Giet \cite{FG04} that if $\lim\limits_{i \rightarrow \infty}{K(i,j) \over i}=\ell(j)>0$, then the hydrodynamic limit of 
the Marcus-Lushnikov process with kernel $K(i,j)$ is the solution to the corresponding modified Smoluchowski (Flory) system.
While, in Jeon \cite{Jeon98} and Norris \cite{Norris99} it was established that $\lim\limits_{i \rightarrow \infty}{K(i,j) \over i}=0$ implies the hydrodynamic limit of 
the Marcus-Lushnikov process is the solution to the Smoluchowski system.

The question whether $T_{gel}<\infty$ is the question of whether the gelation phenomenon occurs in a given system of Smoluchowski equations. The first mathematical proof of gelation was produced in McLeod \cite{McLeod62} for the multiplicative kernel.  Historically, this happened around the time when the formation of a giant cluster 
in the Erd\H{o}s-R\'enyi random graph model (see Sect.~\ref{sec:er}) was proved by P. Erd\H{o}s and A. R\'enyi  \cite{ER60}. The overlap in mathematical formulas obtained in the two papers, \cite{McLeod62} and \cite{ER60}, representing the two different branches of mathematics  is quite remarkable. The work of finding a mathematically solid proof of gelation phenomenon for other conjectured gelling kernels began fifteen years later with the work of Lushnikov \cite{Lushnikov}. It continued with publications of Ziff \cite{Ziff}, Ernst et al. \cite{EHZ}, van Dongen and Ernst \cite{vDE86}, Jeon \cite{Jeon98,Jeon99}, Escobedo et al. \cite{EMP02}, and many other mathematicians and mathematical physicists. 
In Spouge \cite{Spouge83}, gelation is demonstrated numerically for the general bilinear kernel $K(i,j)=A+B(i+j)+Cij$. Aldous \cite{Aldous98} proved gelation for $~K(i,j)={2(ij)^\gamma \over (i+j)^\gamma-i^\gamma-j^\gamma}$, where $\gamma \in (1,2)$. While $\gamma=2$ corresponds to the multiplicative kernel for which, as we know, gelation also occurs. Jeon \cite{Jeon99} proved that complete and instantaneous gelation occurs if $K(i,j) \geq ij\psi(i,j)$, where
$\psi(i,j)$ is a function increasing in both variables, $i$ and $j$, such that 
$~\sum\limits_{j=1}^\infty {1 \over j \psi(i,j)} < \infty~$ for all $i$.
This includes $K(i,j)=(ij)^\alpha$, $\alpha>1$, as a primary example. 
Finally, Rezakhanlou \cite{Rezakhanlou13} lists sufficient conditions for each of the three modes of gelation, i.e., simple, instantaneous, and complete.

\subsection{Multiplicative coalescent and its coagulation equations}\label{sec:mult}
Consider the Marcus-Lushnikov process ${\bf ML}_n(t)$ corresponding to the multiplicative coalescent process that begins with $n$ singletons, i.e., $K(i,j)=ij$ and ${\bf ML}_n(0)=(n,0,0,\hdots)$. In this case, the Smoluchowski coagulation equations \eqref{eqn:GeneralSmoluchowski} are stated as follows
\begin{equation}\label{eqn:multSmol}
{d \over dt}\zeta_k=-k\zeta_k\sum_{j=1}^\infty j\zeta_j+{1 \over 2}\sum_{j=1}^{k-1}j(k-j)\zeta_j\zeta_{k-j} \qquad (k=1,2,\hdots) ~\text{ with } \zeta_k(0)=\delta_{1,k}.
\end{equation}
The dynamics of the total mass $\sum_{j=1}^\infty j\zeta_j$ begins with $\sum_{j=1}^\infty j\zeta_j(0)=1$, and following McLeod \cite{McLeod62}, we have
\begin{align}\label{eqn:diff_mass}
{d \over dt}\sum_{j=1}^\infty j\zeta_j = \sum_{j=1}^\infty j{d \over dt}\zeta_j
& =  -\sum_{i,j=1}^\infty ij^2 \zeta_j \zeta_i +{1 \over 2}\sum_{j=1}^\infty \sum_{i=1}^{j-1}\big(i+(j-i)\big)i(j-i)\zeta_i\zeta_{j-i} \nonumber \\
& =  -\sum_{i,j=1}^\infty ij^2 \zeta_j \zeta_i +{1 \over 2}\sum_{i,j=1}^\infty \big(i+j\big)ij\zeta_i\zeta_j =  0 
\end{align}
provided convergence of $\,\sum\limits_{j=1}^\infty j^2 \zeta_j(t)$. Thus, there exists a time $T_{gel} \in (0,\infty]$, defined as
the time such that the following conservation of mass formula \eqref{eqn:conservation_mass} is satisfied up to $T_{gel}$, i.e., $\,\sum\limits_{j=1}^\infty j \zeta_j(t)=1$.

\bigskip
\noindent
Next, we want to modify the system \eqref{eqn:multSmol} since the decay rate of $k\zeta_k\sum_{j=1}^\infty j\zeta_j$ in \eqref{eqn:multSmol} does not include the gravitation of clusters of size $k$ towards all the rest of the clusters. The problem is that a cluster of an exceptionally large size, say $\epsilon n$, in a single quantity will not be accounted for in \eqref{eqn:multSmol}. Yet, such a large cluster has to contribute $\epsilon k\zeta_k$ to the decay rate of $\zeta_k$. Replacing the decay rate $k\zeta_k\sum_{j=1}^\infty j\zeta_j$ with $k\zeta_k$ would resolve this issue as the new rate accounts for the gravitation of a cluster of a given size $k$ towards all clusters in the Marcus-Lushnikov process, whose weights add up to $n-k= n\big(1+O(n^{-1})\big)$. 

Thus, as it was done in \cite{McLeod62}, the Smoluchowski coagulation system \eqref{eqn:multSmol} reduces to
\begin{equation}\label{eqn:Flory}
{d \over dt}\zeta_k=-k\zeta_k +{1 \over 2}\sum_{j=1}^{k-1}j(k-j)\zeta_j\zeta_{k-j} \qquad (k=1,2,\hdots) ~\text{ with } \zeta_k(0)=\delta_{1,k}
\end{equation}
which is solved explicitly:
\begin{equation}\label{eqn:solutionFlory}
\zeta_k(t)={k^{k-2} t^{k-1} \over k!} e^{-kt} \quad \text{ for } t \geq 0.
\end{equation}
Now, since \eqref{eqn:Flory} is obtained from \eqref{eqn:multSmol} by substituting $\sum_{j=1}^\infty j\zeta_j(t)=1$, 
the solutions of systems \eqref{eqn:multSmol} and \eqref{eqn:Flory} coincide as long as the conservation 
of mass holds, i.e., for all $t \in [0,T_{gel})$.

\medskip
\noindent
The above system of equations \eqref{eqn:Flory} is called {\it modified Smoluchowski system} (see Fournier and Giet \cite{FG04}),
and is also known as the {\it Flory coagulation system} of equations (named after Flory \cite{Flory}).
See also \cite{FL09} for  the analysis of a broad class of Smoluchowski and Flory systems, where the kernel $K(i,j)=i^\alpha j +ij^\alpha$, $\alpha \in (0,1]$.

\medskip
\noindent
Importantly, it is well known that the hydrodynamic limit of  the Marcus-Lushnikov process $~\lim\limits_{n \rightarrow \infty} {\zeta_k^{[n]}(t) \over n}=\zeta_k(t)$ 
is the solution \eqref{eqn:solutionFlory} of the modified Smoluchowski system \eqref{eqn:Flory} for the multiplicative kernel.

\bigskip
\noindent
For the multiplicative kernel $K(i,j)=ij$ the {\it gelation time} $T_{gel}$ is finite \cite{McLeod62}, and therefore, $K(i,j)=ij$ is a {\it gelling kernel}. 
Indeed, applying Stirling's approximation, we have the series
\begin{equation}\label{eqn:multStirling}
\sum\limits_{j=1}^\infty j^2 \zeta_j(s)=\sum\limits_{j=1}^\infty {j^j s^{j-1} \over j!} e^{-js}={1 \over \sqrt{2\pi}s}\sum\limits_{j=1}^\infty {e^{-j(s-\ln{s}-1)} \over \sqrt{j}} \big(1+o(j^{-1})\big)
\end{equation}
converging for all positive $s \not= 1$ and diverging for $s=1$. Hence, the second moment $\,\sum\limits_{j=1}^\infty j^2 \zeta_j(t)$ in \eqref{eqn:diff_mass} converges for $t \in [0,1)$ and diverges for $t=1$, i.e.,
$$T_{gel}:=\inf\Big\{t>0~:~\sum\limits_{j=1}^\infty j^2 \zeta_j(t)=\infty \Big\}=1.$$
Thus, the conservation of mass \eqref{eqn:conservation_mass} is satisfied until the explosion of the second moment at $t=1$.

Moreover, for $t>0$, consider
\begin{equation}\label{defx}
x(t):=\min\{x>0~:~xe^{-x}=te^{-t}\},
\end{equation}
i.e., $x(t)$ is the unique $x\in (0,1]$ such that $xe^{-x}=te^{-t}$. Obviously, $x(t)=t$ for $0< t \leq 1$.

We know that for $t<1$, $\sum\limits_{k=1}^\infty k\zeta_k(t)=1$ implies $\sum\limits_{k=1}^\infty{k^{k-1} (te^{-t})^k \over k!}=t$.
Thus, for all $t>0$, $\sum\limits_{k=1}^\infty{k^{k-1} (te^{-t})^k \over k!}=x(t)$. 
Hence, the first moment of the solutions $\zeta_k$ of \eqref{eqn:Flory} equals
\begin{equation}\label{eqn:xovertMcLeod}
\sum\limits_{k=1}^\infty k\zeta_k(t)={1 \over t}\sum\limits_{k=1}^\infty{k^{k-1} (te^{-t})^k \over k!}={x(t) \over t}.
\end{equation}
and therefore, the solutions of \eqref{eqn:Flory} satisfy
$$\begin{cases}
      \sum\limits_{k=1}^\infty k\zeta_k(t)=1 & \text{ if } t \leq 1, \\
      \sum\limits_{k=1}^\infty k\zeta_k(t)<1 & \text{ if } t > 1. 
\end{cases}$$
Compare this to the  mass conservation property in the Marcus-Lushnikov  processes: $$\sum\limits_{k=1}^\infty k {\zeta_k^{[n]}(t) \over n}=1 \quad \forall t\geq 0.$$
The above can be restated using the following alternative definition of the gelation time
$$T_{gel}:=\inf\Big\{t>0~:~\sum\limits_{j=1}^\infty j \zeta_j(t)<1 \Big\}=1,$$
where $\zeta_k$ are the solutions of \eqref{eqn:Flory}.

\section{The cross-multiplicative coalescent process}\label{sec:knn}
In this section we analyze the cross-multiplicative coalescent process.
We are motivated by the need to extend the theory and applications of coalescent processes to the particle system, described in the introduction, where not all pairs of particles interact with each other. Specifically, each particle may bond only with the particles of the opposite type. 

For given $\alpha, \beta >0$, we consider two integer valued functions, $\alpha[n]=\alpha n +o(\sqrt{n})$ and $\beta[n]=\beta n +o(\sqrt{n})$.
We will examine a coalescent process where the weight of each cluster is a two-dimensional (weight) vector ${\bf i}=\left[\!\!\begin{array}{c}i_1 \\i_2\end{array}\!\!\right]$.
Here, $i_1,i_2 \geq 0$ and $i_1+i_2>0$. 
Each cluster of weight ${\bf i}$ consists of $i_1$ particles of type $A$ and $i_2$ particles of type $B$.  
The coalescent process begins with $\alpha[n]+\beta[n]$ singletons, of which $\alpha[n]$ singletons are of weight $\left[\!\!\begin{array}{c}1 \\0\end{array}\!\!\right]$ and the other $\beta[n]$ singletons are of weight $\left[\!\!\begin{array}{c}0 \\1\end{array}\!\!\right]$. 
The coalescence kernel is defined by 
\begin{equation}\label{biK} 
K({\bf i},{\bf j}):=i_1j_2+i_2j_1
\end{equation}
for any pair of clusters with weight vectors ${\bf i}=\left[\!\!\begin{array}{c}i_1 \\i_2\end{array}\!\!\right]$ and ${\bf j}=\left[\!\!\begin{array}{c}j_1 \\j_2\end{array}\!\!\right]$.
Each pair of clusters of respective weights ${\bf i}$ and ${\bf j}$ would coalesce into a cluster of weight ${\bf i+j}$ with rate $K({\bf i},{\bf j})/n$. The last merger will create a cluster of weight $\left[\!\!\begin{array}{c}\alpha[n] \\ \beta[n]\end{array}\!\!\right]$. We will call this {\it cross-multiplicative coalescent process}, and 
the kernel $K({\bf i},{\bf j})$ defined in \eqref{biK} will be referred to as the {\it cross-multiplicative kernel}.

\subsection{Coagulation equations}
Consider the Marcus-Lushnikov process ${\bf ML}_n(t)$ that keeps track of cluster counts in the above defined cross-multiplicative coalescent process that begins with $\alpha[n]+\beta[n]$ singletons of the two types, $\alpha[n]$ of weight $\left[\!\!\begin{array}{c}1 \\0\end{array}\!\!\right]$ and $\beta[n]$ of weight $\left[\!\!\begin{array}{c}0 \\1\end{array}\!\!\right]$. Specifically, let $\zeta_{i_1,i_2}^{[n]}(t)$ denote the number of components of weight ${\bf i}=\left[\!\!\begin{array}{c}i_1 \\i_2\end{array}\!\!\right]$ at time $t$. Then ${\bf ML}_n(t)$ is the process with coordinates  $\zeta_{i_1,i_2}^{[n]}(t)$, i.e.
$${\bf ML}_n(t)=\Big(\zeta_{i_1,i_2}^{[n]}(t) \Big)_{i_1,i_2}$$
with the starting values $\zeta_{1,0}^{[n]}(0)=\alpha[n]$, $\zeta_{0,1}^{[n]}(0)=\beta[n]$, and  $\zeta_{i_1,i_2}^{[n]}(0)=0$ for all other pairs $(i_1,i_2)$.

\medskip
\noindent
The Smoluchowski coagulation equations for the Marcus-Lushnikov process ${\bf ML}_n(t)$ with cross-multiplicative kernel are written as follows:
\begin{align}\label{eqn:Knn}
{d \over dt}\zeta_{i_1,i_2}(t)=-\zeta_{i_1,i_2}(t)\sum\limits_{j_1,j_2} (i_1j_2+i_2j_1)\zeta_{j_1,j_2}(t)+ {1 \over 2}\sum\limits_{\substack{\ell_1, k_1:~ \ell_1+k_1=i_1,\\ \ell_2, k_2:~ \ell_2+k_2=i_2}} (\ell_1 k_2+\ell_2 k_1)\zeta_{\ell_1,\ell_2}(t)\zeta_{k_1,k_2}(t) \nonumber\\
\end{align}
with the initial conditions $~\zeta_{i_1,i_2}(0)=\alpha \delta_{1,i_1}\delta_{0,i_2}+\beta \delta_{0,i_1}\delta_{1,i_2}$.

\medskip
\noindent
A reduced system of differential equations corresponding to the above Smoluchowski coagulation equations \eqref{eqn:Knn} will be given in \eqref{eqn:Knnt}. It will take into account the mass conservation property of the above Marcus-Lushnikov process ${\bf ML}_n(t)$, and therefore will represent the smaller cluster dynamics over the whole time interval $[0,\infty)$.

\bigskip
\noindent
First, we notice that here the initial total mass is $\sum\limits_{i_1,i_2} (i_1+i_2)\zeta_{i_1,i_2}(0)=\alpha+\beta$. Moreover, the initial total `left mass' (type $A$) is 
$\sum\limits_{i_1,i_2} i_1\zeta_{i_1,i_2}(0)=\alpha$ and the initial total `right mass' (type $B$) is $ \sum\limits_{i_1,i_2} i_2\zeta_{i_1,i_2}(0)=\beta$.

Next, we consider the rate of change for the total left mass and the total right mass, and use \eqref{eqn:Knn} to obtain
\begin{eqnarray}\label{eqn:firstmoment_i1}
{d \over dt}\sum\limits_{i_1,i_2} i_1 \zeta_{i_1,i_2}(t) & = & -\sum\limits_{i_1,i_2,j_1,j_2} i_1(i_1j_2+i_2j_1)\zeta_{i_1,i_2}(t) \zeta_{j_1,j_2}(t) \nonumber \\
& & \quad + {1 \over 2}\sum\limits_{\ell_1, k_1,\ell_2, k_2} (\ell_1+k_1)(\ell_1 k_2+\ell_2 k_1)\zeta_{\ell_1,\ell_2}(t)\zeta_{k_1,k_2}(t) ~=0
\end{eqnarray}
and
\begin{eqnarray}\label{eqn:firstmoment_i2}
{d \over dt}\sum\limits_{i_1,i_2} i_2\zeta_{i_1,i_2}(t) & = & -\sum\limits_{i_1,i_2,j_1,j_2} i_2(i_1j_2+i_2j_1)\zeta_{i_1,i_2}(t) \zeta_{j_1,j_2}(t) \nonumber \\
& & \quad + {1 \over 2}\sum\limits_{\ell_1, k_1,\ell_2, k_2} (\ell_2+k_2)(\ell_1 k_2+\ell_2 k_1)\zeta_{\ell_1,\ell_2}(t)\zeta_{k_1,k_2}(t)~=0
\end{eqnarray}
whenever $\sum\limits_{i_1,i_2} (i_1+i_2)^2\zeta_{i_1,i_2}(t)$ converges. 

\bigskip
\noindent
Here, for $t<T_{gel}$, $\sum\limits_{j_1,j_2} j_1\zeta_{j_1,j_2}(t)=\alpha$ and \mbox{$\sum\limits_{j_1,j_2} j_2\zeta_{j_1,j_2}(t)=\beta$}. Therefore, for any $i_1$ and $i_2$,
$~\sum\limits_{j_1,j_2} (i_1j_2+i_2j_1)\zeta_{j_1,j_2}(t)=\beta i_1+\alpha i_2.$
Thus we can consider the following {\it modified Smoluchowski coagulation system of equations}:
\begin{equation}\label{eqn:Knnt}
{d \over dt}\zeta_{i_1,i_2}(t)=-(\beta i_1+ \alpha i_2)\zeta_{i_1,i_2}(t)+ {1 \over 2}\sum\limits_{\substack{\ell_1, k_1:~ \ell_1+k_1=i_1,\\ \ell_2, k_2:~ \ell_2+k_2=i_2}} (\ell_1 k_2+\ell_2 k_1)\zeta_{\ell_1,\ell_2}(t)\zeta_{k_1,k_2}(t)
\end{equation}
with the initial conditions $~\zeta_{i_1,i_2}(0)=\alpha \delta_{1,i_1}\delta_{0,i_2}+\beta \delta_{0,i_1}\delta_{1,i_2}$. Once again, the solutions of Smoluchowski coagulation system \eqref{eqn:Knn} and the above modified Smoluchowski coagulation system (\ref{eqn:Knnt}) will match up until $T_{gel}$. Consequently, the solution \eqref{eqn:zetasol} of the modified Smoluchowski system of equations \eqref{eqn:Knnt} is used in Sect.~\ref{sec:gelKnn} for establishing the finiteness of the gelation time and for finding its value, $T_{gel}$.

\medskip
\noindent
In Sect.~\ref{sec:hydcross} we establish that the solution to the above modified Smoluchowski coagulation system \eqref{eqn:Knnt} is 
the hydrodynamic limit of the Marcus-Lushnikov process ${\bf ML}_n(t)$ with cross-multiplicative kernel. 
Specifically, in equation \eqref{crossAS}, it is shown that
$$\lim\limits_{n \to \infty} \sup\limits_{s \in [0,T]} \left|n^{-1}\zeta_{i_1,i_2}^{[n]}(s)-\zeta_{i_1,i_2}(s)\right|=0 \qquad \text{ a.s.}$$
for any given $T>0$ and all $i_1, i_2 \geq 1$, where $\zeta_{i_1,i_2}(t)$ solves the modified Smoluchowski coagulation system \eqref{eqn:Knnt}.

\subsection{The unique solution of the modified Smoluchowski coagulation system} \label{sub:sol}
Next, we want to find the solution $\zeta_{i_1,i_2}(t)$ of the reduced system \eqref{eqn:Knnt} for all $t \geq 0$.
Here we observe that $\zeta_{1,0}(t)=\alpha e^{-\beta t}$ and $\zeta_{0,1}(t)=\beta e^{-\alpha t}$, and extend the approach of McLeod \cite{McLeod62} by considering the solutions of the following form
\begin{equation}\label{eqn:sol}
\zeta_{i_1,i_2}(t)=\alpha^{i_1}\beta^{i_2}S_{i_1,i_2}e^{-(\beta i_1+ \alpha i_2)t}t^{i_1+i_2-1}
\end{equation}
and plugging them into  equation (\ref{eqn:Knnt}). After cancelations, we arrive with the following recursion
\begin{equation}\label{eqn:recS}
(i_1+i_2-1)S_{i_1,i_2}={1 \over 2}\sum\limits_{\substack{\ell_1, k_1:~ \ell_1+k_1=i_1,\\ \ell_2, k_2:~ \ell_2+k_2=i_2}} (\ell_1 k_2+\ell_2 k_1)S_{\ell_1,\ell_2}S_{k_1,k_2}
\end{equation}
with initial conditions $S_{i,0}=S_{0,i}=\delta_{1,i}$, and $S_{i_1,i_2}=S_{i_2,i_1}$.

\medskip
\noindent
In the next lemma we state the explicit solution to the recursion relation (\ref{eqn:recS}) which we prove using a generalization of Abel's binomial theorem.
\begin{lemma}\label{lem:S}
The system of equations (\ref{eqn:recS}) with the initial conditions $S_{i,0}=S_{0,i}=\delta_{1,i}$ has the following unique solution
\begin{equation}\label{eqn:solS}
S_{i_1,i_2}={i_1^{i_2-1}i_2^{i_1-1} \over i_1!i_2!}.
\end{equation}
\end{lemma}

\medskip
\noindent
Note that the numerator $i_1^{i_2-1}i_2^{i_1-1}$ in \eqref{eqn:solS} is the total number of spanning trees in $K_{i_1,i_2}$. See \cite{Austin60}.

\begin{proof}
In Theorem 1.1(3) of \cite{HL}, F. Huang and B. Liu generalize Abel's binomial theorem as follows:
\begin{align}\label{eqn:HL}
\sum\limits_{k_1=0}^{i_1} \sum\limits_{k_2=0}^{i_2} & \binom{i_1}{k_1}\binom{i_2}{k_2}\big(v+zi_1-zk_1\big)^{k_2-1} \big(-z(i_1-k_1)\big)^{i_2-k_2}(-zk_2)^{k_1} (u+zk_2)^{i_1-k_1-1}  \nonumber \\
&= {[uv-i_1i_2z^2]u^{i_1-1}v^{i_2-1} \over (v+i_1z)(u+i_2z)}
\end{align}
Then, we use (\ref{eqn:HL}) with $z=-1$ to confirm our candidate solution satisfies (\ref{eqn:recS}) by plugging it into the right hand side of (\ref{eqn:recS}) as follows.
\begin{align*}
{1 \over 2}&\sum\limits_{\substack{\ell_1, k_1:~ \ell_1+k_1=i_1,\\ \ell_2, k_2:~ \ell_2+k_2=i_2}}  (\ell_1 k_2+\ell_2 k_1)S_{\ell_1,\ell_2}S_{k_1,k_2} 
\quad =\sum\limits_{\substack{\ell_1, k_1:~ \ell_1+k_1=i_1,\\ \ell_2, k_2:~ \ell_2+k_2=i_2}} \ell_1 k_2 S_{\ell_1,\ell_2}S_{k_1,k_2} \\
&=\sum\limits_{\substack{\ell_1, k_1:~ \ell_1+k_1=i_1, \\ \ell_2, k_2:~ \ell_2+k_2=i_2, \\ (k_1,k_2), (\ell_1,\ell_2)\not=(0,0)}} {\ell_1^{\ell_2} \ell_2^{\ell_1-1} k_1^{k_2-1} k_2^{k_1} \over \ell_1!\ell_2!k_1!k_2!} 
\end{align*}
\begin{align*}
\quad &={1 \over i_1!i_2!} \sum\limits_{\substack{k_1:~ 0 \leq k_1 \leq i_1, \\ k_2:~ 0 \leq k_2 \leq i_2, \\ (k_1,k_2)\not=(0,0), (i_1,i_2)}} \binom{i_1}{k_1}\binom{i_2}{k_2}k_1^{k_2-1} (i_1-k_1)^{i_2-k_2}k_2^{k_1} (i_2-k_2)^{i_1-k_1-1} \\
&= {1 \over i_1!i_2!} \lim\limits_{\substack{v  \rightarrow i_1,\\u \rightarrow i_2}} \Big\{\sum\limits_{k_1=0}^{i_1}  \sum\limits_{k_2=0}^{i_2}  \binom{i_1}{k_1}\binom{i_2}{k_2}\big(v-i_1+k_1\big)^{k_2-1} (i_1-k_1)^{i_2-k_2}k_2^{k_1} (u-k_2)^{i_1-k_1-1}\\
& \qquad \qquad \qquad \qquad \qquad \qquad \qquad \qquad \qquad \qquad \qquad \qquad \qquad -{i_1^{i_2}u^{i_1-1} \over v-i_1}-{i_2^{i_1}v^{i_2-1} \over u-i_2} \Big\}
\end{align*}
\begin{align*}
\quad &={1 \over i_1!i_2!} \lim\limits_{\substack{v  \rightarrow i_1,\\u \rightarrow i_2}}  \Big\{ {[uv-i_1i_2]u^{i_1-1}v^{i_2-1} \over (v-i_1)(u-i_2)}-{i_1^{i_2}u^{i_1-1} \over v-i_1}-{i_2^{i_1}v^{i_2-1} \over u-i_2} \Big\}\\
&={1 \over i_1!i_2!} \lim\limits_{\substack{v  \rightarrow i_1,\\u \rightarrow i_2}}  \Big\{ {i_1v^{i_2-1}u^{i_1-1} \over v-i_1}+{u^{i_1}v^{i_2-1} \over u-i_2}-{i_1^{i_2}u^{i_1-1} \over v-i_1}-{i_2^{i_1}v^{i_2-1} \over u-i_2} \Big\}
\end{align*}
Hence,
\begin{align*}
{1 \over 2}\sum\limits_{\substack{\ell_1, k_1:~ \ell_1+k_1=i_1,\\ \ell_2, k_2:~ \ell_2+k_2=i_2}}  (\ell_1 k_2+\ell_2 k_1)S_{\ell_1,\ell_2}S_{k_1,k_2} &={1 \over i_1!i_2!} \lim\limits_{\substack{v  \rightarrow i_1,\\u \rightarrow i_2}}  \Big\{ i_1u^{i_1-1}{v^{i_2-1}-i_1^{i_2-1} \over v-i_1}+v^{i_2-1} {u^{i_1} -i_2^{i_1} \over u-i_2} \Big\}\\
&={1 \over i_1!i_2!} \big((i_2-1)\cdot i_1^{i_2-1}i_2^{i_1-1}+i_1\cdot i_1^{i_2-1}i_2^{i_1-1}\big)\\
&= (i_1+i_2-1){i_1^{i_2-1}i_2^{i_1-1} \over i_1!i_2!} \\
&= (i_1+i_2-1)S_{i_1,i_2} 
\end{align*}
thus completing the proof.
\end{proof}

\noindent
The solution of equations (\ref{eqn:Knnt}) follows from (\ref{eqn:sol}) and Lemma \ref{lem:S}.
\begin{thm}\label{thm:solutionKnnt}
The modified Smoluchowski coagulation system of equations \eqref{eqn:Knnt} with the initial conditions $~\zeta_{i_1,i_2}(0)=\alpha \delta_{1,i_1}\delta_{0,i_2}+\beta \delta_{0,i_1}\delta_{1,i_2}$ has the unique solution
\begin{equation}\label{eqn:zetasol}
\zeta_{i_1,i_2}(t)={i_1^{i_2-1}i_2^{i_1-1}\alpha^{i_1}\beta^{i_2} \over i_1!i_2!}e^{-(\beta i_1+ \alpha i_2)t}t^{i_1+i_2-1}.
\end{equation}
\end{thm}

\bigskip

\subsection{Gelation in the cross-multiplicative coalescent process}\label{sec:gelKnn}
Next, we prove the finiteness of the gelation time that, following the approach in \eqref{eqn:gelationGeneral02}, we define as 
$$T_{gel}:=\inf\Big\{t>0~:~\sum\limits_{i_1,i_2} (i_1+i_2)\zeta_{i_1,i_2}(t)<\alpha+\beta \Big\}.$$

\noindent
Let 
\begin{equation}\label{eqn:suv}
s(u,v):=\sum\limits_{(i_1,i_2)\in \mathbb{Z}_+^2 \setminus \{(0,0)\}} S_{i_1,i_2} u^{i_1}v^{i_2} =\sum\limits_{(i_1,i_2)\in \mathbb{Z}_+^2 \setminus \{(0,0)\}} {i_1^{i_2-1}i_2^{i_1-1} \over i_1!i_2!} u^{i_1}v^{i_2}
\end{equation}
be the generating function of $S_{i_1,i_2}$.
The recurrence relation \eqref{eqn:recS} implies
\begin{equation}\label{eqn:recSode}
u {\partial s \over \partial u}+v{\partial s \over \partial v} -s=u v {\partial s \over \partial u}{\partial s \over \partial v}
\end{equation}
with the initial conditions ${\partial \over \partial u}s(0,1)={\partial \over \partial v}s(1,0)=1$.

\begin{lemma}\label{lem:GelationTimeKanbn}
Consider the Smoluchowski coagulation system of equations  \eqref{eqn:Knn} with the initial conditions $~\zeta_{i_1,i_2}(0)=\alpha \delta_{1,i_1}\delta_{0,i_2}+\beta \delta_{0,i_1}\delta_{1,i_2}$. Then, a phase transition occurs at
$$\inf\Big\{t>0~:~\sum\limits_{i_1,i_2} (i_1+i_2)^2\zeta_{i_1,i_2}(t)=\infty \Big\}={1 \over \sqrt{\alpha \beta}}.$$ 
\end{lemma}

\noindent
Note that the above phase transition corresponds to the gelation times as defined in \eqref{eqn:gelationGeneral01}.
\begin{proof}
We will follow the approach in \cite{Aldous98,Aldous} and \cite{Ziff}. Let 
$$E(t):=\sum\limits_{i_1,i_2} i_1^2\zeta_{i_1,i_2}(t), \quad F(t):=\sum\limits_{i_1,i_2} i_1 i_2\zeta_{i_1,i_2}(t), 
\quad \text{ and }\quad G(t):=\sum\limits_{i_1,i_2} i_2^2\zeta_{i_1,i_2}(t)$$
denote all the second order moments of $\zeta_{i_1,i_2}(t)$.
By differentiating as in \eqref{eqn:firstmoment_i1} and \eqref{eqn:firstmoment_i2}, we obtain
$${d \over dt}E(t)=2E(t)F(t), \quad {d \over dt}F(t)=E(t)G(t)+F^2(t), \quad \text{ and }\quad {d \over dt}G(t)=2G(t)F(t)$$
with the initial conditions $E(0)=\alpha$, $F(0)=0$, and $G(0)=\beta$.
We require the finiteness of all third order moments when deriving the above differential equations for the second order moments.
Here the first and the third equations yield $E(t)={\alpha \over \beta}G(t)$. 
Hence the system reduces to
$${d \over dt}E(t)=2E(t)F(t)  \quad \text{ and }\quad {d \over dt}F(t)={\beta \over \alpha} E^2(t)+F^2(t),$$
and therefore,
$${d \over dt}\left(\sqrt{\beta \over \alpha}E(t)+F(t)\right)=\left(\sqrt{\beta \over \alpha}E(t)+F(t)\right)^2.$$
Thus,
$$\sqrt{\beta \over \alpha}E(t)+F(t)={1 \over {1 \over \sqrt{\alpha\beta}}-t}$$
for $t<{1 \over \sqrt{\alpha \beta}}$. The statement of the lemma follows from the fact that all functions obtained as all-order partial derivatives of the series \eqref{eqn:suv} 
have the same domain of convergence.
\end{proof}

\bigskip
\noindent
For given $\alpha,\beta >0$ and $t > 0$, define
\begin{equation}\label{defxy}
\big(x(t),y(t)\big):=\min\big\{(x,y)~:~xe^{-y}=\alpha te^{-\beta t},~ye^{-x}=\beta te^{-\alpha t}\big\},
\end{equation}
where the minimum in one coordinate implies the minimum in another as $x$ and $y$ solving
\begin{equation}\label{eqn:xey}
xe^{-y}=u \quad \text{ and } \quad ye^{-x}=v
\end{equation}
for $u,v>0$ are mutually monotonous, e.g. $x=ue^y$.

\begin{prop}\label{prop:xyvsone}
For given $u,v>0$, consider the system \eqref{eqn:xey}. Then, the following holds.
\begin{itemize}
  \item[(i)] Depending on the values of $u$ and $v$, the system \eqref{eqn:xey}  may have one, two, or no solutions.
  \item[(ii)] If the system \eqref{eqn:xey} has a unique solution, then the solution should satisfy $xy=1$.
  \item[(iii)] If the system \eqref{eqn:xey} has two solutions, then the smallest solution should satisfy $xy < 1$, and the largest solution should satisfy $xy>1$.
\end{itemize}
\end{prop}
\begin{proof}
First, observe that $x=ue^{ve^x}$, and statement (i) follows from the convexity of $ue^{ve^x}$.

\medskip
\noindent
Next, suppose $(x_1,y_1)$ and $(x_2,y_2)$ are two solutions of \eqref{eqn:xey}. Then
\begin{equation}\label{eqn:x1y1x2y2}
x_1e^{-y_1}=x_2 e^{-y_2}\quad \text{ and } \quad y_1e^{-x_1}=y_2e^{-x_2}.
\end{equation}
We express $x_2$ in terms of $x_1$ and $y_1$, obtaining $x_2=x_1e^{y_1(e^{x_2-x_1}-1)}$.
We notice that there is a unique solution $x=x_1$ of
\begin{equation}\label{eqn:xviax1y1}
x=x_1e^{y_1(e^{x-x_1}-1)}
\end{equation}
if and only if $x=x_1$ is the root of $1=x y_1e^{x-x_1}$. This yields statement (ii).

Finally, suppose there are two distinct solutions of \eqref{eqn:xviax1y1}, and $x_2>x_1$ (implying $y_2>y_1$).
Then, there is a local extremum $x \in (x_1,x_2)$, satisfying
$1=x y_1e^{x-x_1}>x_1y_1$.

Similarly, suppose there are two distinct solutions of \eqref{eqn:xviax1y1}, and $x_2<x_1$ (implying $y_2<y_1$).
Then, there is a local extremum $x \in (x_2,x_1)$, satisfying
$1=x y_1e^{x-x_1}<x_1y_1$. Hence, statement (iii).
\end{proof}

Prop.~\ref{prop:xyvsone} immediately yields the following corollary concerning the functions defined in \eqref{defxy}.
\begin{cor}\label{cor:xy1}
For given $\alpha,\beta >0$ and $t > 0$, consider $\big(x(t),y(t)\big)$ as defined in \eqref{defxy}. Then, 
\begin{itemize}
  \item $x(t)=\alpha t$ and $y(t)=\beta t$ for all $t \leq {1 \over \sqrt{\alpha \beta}}$;
  \item $x(t)<\alpha t$ and $y(t)<\beta t$ for all $t > {1 \over \sqrt{\alpha \beta}}$.
  \end{itemize}
\end{cor}

\medskip
\noindent
Next, we derive an analogue to the equation \eqref{eqn:xovertMcLeod} that was proved in McLeod \cite{McLeod62} for the regular multiplicative kernel. 
\begin{lemma}\label{lem:novoMcLeod}
Consider the solution $\zeta_{i_1,i_2}(t)$ of the modified Smoluchowski coagulation system of equations \eqref{eqn:Knnt} with the initial conditions $~\zeta_{i_1,i_2}(0)=\alpha \delta_{1,i_1}\delta_{0,i_2}+\beta \delta_{0,i_1}\delta_{1,i_2}$, as found in Theorem \ref{thm:solutionKnnt}. Then,
$$\sum\limits_{i_1,i_2} i_1\zeta_{i_1,i_2}(t)=\sum\limits_{i_1,i_2} {i_1^{i_2}i_2^{i_1-1}\alpha^{i_1}\beta^{i_2} \over i_1!i_2!}e^{-(\beta i_1+ \alpha i_2)t}t^{i_1+i_2-1}={x(t) \over t}$$
and
$$\sum\limits_{i_1,i_2} i_2\zeta_{i_1,i_2}(t)=\sum\limits_{i_1,i_2} {i_1^{i_2-1}i_2^{i_1}\alpha^{i_1}\beta^{i_2} \over i_1!i_2!}e^{-(\beta i_1+ \alpha i_2)t}t^{i_1+i_2-1}={y(t) \over t},$$
where $x(t)$ and $y(t)$ are the functions defined in \eqref{defxy}.
\end{lemma}
\noindent
\begin{proof}
By \eqref{eqn:firstmoment_i1}, \eqref{eqn:firstmoment_i2}, and Lemma~\ref{lem:GelationTimeKanbn}, we have
\begin{equation}\label{eqn:atbtdsuv}
\alpha t=\alpha t e^{-\beta t}{\partial s \over \partial u}\big(\alpha t e^{-\beta t},\beta t e^{-\alpha t}\big)~~\text{ and }~~\beta t=\beta t e^{-\alpha t}{\partial s \over \partial v}\big(\alpha t e^{-\beta t},\beta t e^{-\alpha t}\big) 
\end{equation}
$\forall \alpha,\beta >0$ and $\forall t < {1 \over \sqrt{\alpha \beta}}$.
Now, since the function $s(u,v)$ does not depend on the values of $\alpha$ and $\beta$, \eqref{eqn:atbtdsuv} implies
$$x=xe^{-y}{\partial s \over \partial u}\big(xe^{-y},ye^{-x}\big)  \quad \text{ and } \quad y=ye^{-x}{\partial s \over \partial v}\big(xe^{-y},ye^{-x}\big)$$
for all $xy<1$. Hence, by Prop.~\ref{prop:xyvsone} , we have
$$x=u{\partial s \over \partial u}(u,v)  \quad \text{ and } \quad y=v{\partial s \over \partial v}(u,v)$$
whenever $(x,y)$ is the smallest solution of \eqref{eqn:xey}.
The equations
\begin{equation}\label{eqn:xtytdsuv}
x(t)=\alpha t e^{-\beta t}{\partial s \over \partial u}\big(\alpha t e^{-\beta t},\beta t e^{-\alpha t}\big)~~\text{ and }~~y(t)=\beta t e^{-\alpha t}{\partial s \over \partial v}\big(\alpha t e^{-\beta t},\beta t e^{-\alpha t}\big) \qquad \forall t \geq 0,
\end{equation}
with $x(t)$ and $y(t)$ defined in \eqref{defxy}, follow from Corollary \ref{prop:xyvsone}.
\end{proof}

\begin{cor}\label{cor:GelationTimeKanbn}
The cross-multiplicative kernel defined in \eqref{biK} is a gelling kernel, and the gelation time corresponding to the Smoluchowski coagulation system of equations 
\eqref{eqn:Knn} with the initial conditions $~\zeta_{i_1,i_2}(0)=\alpha \delta_{1,i_1}\delta_{0,i_2}+\beta \delta_{0,i_1}\delta_{1,i_2}$ equals
$$T_{gel}={1 \over \sqrt{\alpha \beta}}.$$ 
\end{cor}
\begin{proof}
Lemma \ref{lem:novoMcLeod} and Corollary \ref{prop:xyvsone} imply that the mass of the system in \eqref{eqn:Knnt} is conserved until ${1 \over \sqrt{\alpha \beta}}$, after which time it begins to dissipate, i.e.,
$$\begin{cases}
      \sum\limits_{i_1,i_2} (i_1+i_2)\zeta_{i_1,i_2}(t)=\alpha+\beta & \text{ if } t \leq {1 \over \sqrt{\alpha \beta}}; \\
      \sum\limits_{i_1,i_2} (i_1+i_2)\zeta_{i_1,i_2}(t)<\alpha+\beta & \text{ if } t > {1 \over \sqrt{\alpha \beta}}.
\end{cases}$$
\end{proof}

\medskip
\noindent
Recall that we considered two alternative definitions of gelation time in Sect.~\ref{sec:gel}. Definition \eqref{eqn:gelationGeneral01} would often describe the time of the explosion of a higher moment while definition \eqref{eqn:gelationGeneral02} is based on the loss of total mass after gelation.
Comparing Lemma \ref{lem:GelationTimeKanbn} with Corollary \ref{cor:GelationTimeKanbn}, we confirm the equivalence of the two alternative definitions of the gelation time
for the cross-multiplicative kernel, i.e.,
$$\inf\Big\{t>0~:~\sum\limits_{i_1,i_2} (i_1+i_2)^2\zeta_{i_1,i_2}(t)=\infty \Big\}=T_{gel}=\inf\Big\{t>0~:~\sum\limits_{i_1,i_2} (i_1+i_2)\zeta_{i_1,i_2}(t)<\alpha+\beta \Big\}.$$

\section{Applications in minimal spanning trees}\label{sec:applications} 

In this section we demonstrate how the coagulation equations for the multiplicative and cross-multiplicative coalescent processes and 
the weak convergence results of Section \ref{sec:weakconv} can be used for finding the lengths of the minimal spanning trees 
on the complete graph $K_n$ and on the complete bipartite graph $K_{\alpha[n], \beta[n]}$ respectively. 
The main result of this section Theorem \ref{thm:main} is proved using Marcus-Lushnikov processes and coagulation equations in Sect.~\ref{sec:proofKnn}.
The proof of Theorem \ref{thm:main} is preceded by the proof of Theorem \ref{thm:MLmult} in Sect.~\ref{sec:proofKn}, a well-known result of Frieze \cite{Frieze}
that we use to illustrate the approach.

\medskip
\noindent
We recall the following quote from Aldous \cite{Aldous98}:  {\it It turns out that there is a large scientific literature relevant to the Marcus-Lushnikov process, mostly focusing on its deterministic approximation. Curiously, this literature has been largely ignored by random graph theorists}. The broader goal of this section is in bridging the gap between the theory of the Smoluchowski coagulation equations for the Marcus-Lushnikov processes and the random graph theory. 
Here we concentrate on analyzing the length of the minimal spanning tree as the prime example that demonstrates the usefulness of the Marcus-Lushnikov processes and the coalescence theory in general for answering questions about random graphs.  We recall that the asymptotic limit for the mean length of a minimal spanning tree on $K_n$ with independent uniform edge weights over $[0, 1]$ was derived in Frieze \cite{Frieze}. There, it is shown to be $~\lim\limits_{n \rightarrow \infty} E[L_n]=\zeta(3)=\sum\limits_{k=1}^\infty {1 \over k^3}$. 
The mean length of a minimal spanning tree on the complete bipartite graph $K_{n,n}$ with independent uniform edge weights over $[0, 1]$ was shown in \cite{FMcD} to be  $~\lim\limits_{n \rightarrow \infty} E[L_n]=2\zeta(3)$. 
In Beveridge et al \cite{BFMcD}, the minimal spanning tree problem was addressed for $d$-regular graphs. In Sect.~\ref{sec:LnKnn} and \ref{sec:proofKnn}, we will find the mean length of the minimal spanning tree in the case of a complete bipartite graph $K_{\alpha[n], \beta[n]}$ via a connection between the coalescence theory and the random graph theory. Note that $K_{\alpha[n], \beta[n]}$ is an irregular graph when $\alpha \not= \beta$.

\subsection{Relation of Erd\H{o}s-R\'enyi process on $K_n$ to multiplicative coalescent}\label{sec:er}
Recall that Erd\H{o}s-R\'enyi random graph is a model on a complete graph of $n$ vertices, $K_n$, where each edge $e$ of $\binom{n}{2}$ edges there is an associated uniform random variable $U_e$ over $[0,1]$. The random variables $\{U_e\}_e$ are assumed to be independent. For the ``time" parameter $p \in [0,1]$, an edge $e$ is considered ``open" if $U_e \leq p$. Erd\H{o}s-R\'enyi random graph $G(n,p)$ will consist of all $n$ vertices and all open edges at time $p$. The number of open edges is a binomial random variable with parameters $\binom{n}{2}$ and $p$, and mean value $\binom{n}{2}p \sim {pn^2 \over 2}$. As we increase $p$, more and more edges open up, new clusters are created, and cluster merges occur. Thus Erd\H{o}s-R\'enyi random graph model can be viewed as a dynamical model that describes an evolution of a random graph \cite{ER60}. 

\medskip
\noindent
If we condition on the number of edges in $G(n,p)$, the graph structure will no longer depend on $p$. Let $\xi_{n,N}$ be the number of components in an Erd\H{o}s-R\'enyi random graph with $n$ vertices and $N$ edges. For $t>0$, letting $N \sim {tn \over 2}$, Theorem 6 in \cite{ER60} by P. Erd\H{o}s and A. R\'enyi states that
\begin{equation}\label{ER6}
{E[\xi_{n,N}] \over n}={1 \over t}\sum\limits_{k=1}^\infty{k^{k-2} (te^{-t})^k \over k!}+\mathcal{R}_t,
\end{equation}
where the error term is
$$\mathcal{R}_t=\begin{cases}
      O\left({1 \over n}\right) & \text{ if } 0<t<1 \\
      O\left({\log{n} \over n}\right) & \text{ if } t=1 \\
      o(1) & \text{ if } t>1 
\end{cases}.$$
There $\varphi(t)={1 \over t}\sum\limits_{k=1}^\infty{k^{k-2} (te^{-t})^k \over k!}$ reaches its maximum at $t=1$, and
$~\varphi(1)=\sum\limits_{k=1}^\infty{k^{k-2} e^{-k} \over k!}={1 \over 2}$.
Recall the function $x(t)$ defined in \eqref{defx}. 
It was pointed out by P. Erd\H{o}s and A. R\'enyi that $\varphi(t)$ in the equation \eqref{ER6} can be represented via $x(t)$ as follows,
$$\varphi(t)={x(t)-{x^2(t) \over 2} \over t}.$$
Observe that here, since we are letting $N \sim {tn \over 2}$, parameter $t$ is essentially equivalent to $np$. So $t$ is a scaled time parameter.

\bigskip
\noindent
Let  the number of vertices in a connected component of a random graph be referred to as a weight of the cluster (or cluster size). 
In Sect.~\ref{sec:mult} we considered the Marcus-Lushnikov process 
$${\bf ML}_n(t)=\Big(\zeta_1^{[n]}(t),\zeta_2^{[n]}(t),\hdots, \zeta_n^{[n]}(t), 0,0,\hdots \Big)$$
corresponding to the multiplicative coalescent process of $n$ particles that begins with $n$ singletons,  i.e., ${\bf ML}_n(0)=(n,0,0,\hdots)$. As observed in \cite{Aldous}, the process ${\bf ML}_n(t)$ describes cluster size dynamics of the Erd\H{o}s-R\'enyi random graph process $G(n,p)$ with $p=1-e^{-t/n}$. Here the scaled time parameter in the Erd\H{o}s-R\'enyi process is $np=n\big(1-e^{-t/n}\big) \sim t$. Thus the time scale is consistent with the one used in \cite{ER60} by P. Erd\H{o}s and A. R\'enyi.

\bigskip
\noindent
Recall the function $\zeta_k$ in \eqref{eqn:solutionFlory} that solves \eqref{eqn:Flory}. 
As we know, in the Erd\H{o}s-R\'enyi process, the gelation time $T_{gel}=1$ of the Marcus-Lushnikov process with the multiplicative kernel 
corresponds to a time after which a single giant component emerges, and continues to absorb components of smaller size. Indeed, in \cite{ER60}, P. Erd\H{o}s and A. R\'enyi showed that the cycles are rare for a given fixed $t > 0$, and the clusters of size $k$ at time $t$ consist mainly of {\it isolated trees of order} $k$. Specifically, if $\tau_k$ denotes the number of isolated trees of order $k$, Theorem 4b in  \cite{ER60} asserts that
\begin{equation}\label{eqn:er1}
\lim\limits_{n \rightarrow \infty} {kE[\tau_k] \over n}={k^{k-2} t^{k-1} \over k!} e^{-kt}=\zeta_k(t)
\end{equation}
and
\begin{equation}\label{eqn:er2}
\lim\limits_{n \rightarrow \infty} {\sum\limits_{k=1}^\infty kE[\tau_k] \over n}=\lim\limits_{n \rightarrow \infty} {\sum\limits_{k=1}^n kE[\tau_k] \over n}={x(t) \over t},
\end{equation}
where $x(t)$ is defined in \eqref{defx}.
Moreover, Theorem 9b in  \cite{ER60} proves the emergence of one giant component after time $t=1$. There, if we let $\gamma_n(t)$ denote the size of the greatest component at time $t$, then 
$$\lim\limits_{n \rightarrow \infty} {\gamma_n(t) \over n}=1-{x(t) \over t} \qquad \text{ in probability.}$$
So the dynamics of $~g(t):=1-\sum\limits_{k=1}^\infty k\zeta_k(t)=1-{x(t) \over t}~$ represents the asymptotic size of the giant component.


\subsection{The length of the minimal spanning tree in $K_n$}\label{lmstKn}
Recall that in the construction of the Erd\H{o}s-R\'enyi random graph model, each edge $e$ of the complete graph $K_n$ had a random variable $U_e$ associated with it. Here we consider $U_e$ to be uniform over $[0,1]$. However, in general, various types of probability distributions are considered in the extensive literature on the topic.  Now, thinking of $U_e$ as the length of the edge $e$, one can construct a minimal spanning tree on $K_n$. Let random variable $L_n$ denote the length of such minimal spanning tree. The asymptotic limit of the mean value of $L_n$ was considered in Frieze \cite{Frieze}. There, the results \eqref{eqn:er1} and \eqref{eqn:er2} from  P. Erd\H{o}s and A. R\'enyi \cite{ER60} are used in proving the following limit
\begin{equation}\label{eqn:frieze1}
\lim\limits_{n \rightarrow \infty} E[L_n]=\int\limits_0^\infty {x(t) \over t} dt=\sum\limits_{k=1}^\infty \int\limits_0^\infty {k^{k-2} t^{k-1} \over k!} e^{-kt} dt=\zeta(3),
\end{equation}
where $\zeta(3)=\sum\limits_{k=1}^\infty {1 \over k^3}=1.202\hdots$ is the value of the Riemann zeta function at $3$.

Consider a coalescent process with a kernel $K(i,j)$ for which $T_{gel}<\infty$ has been proved. See \cite{Aldous98, Jeon99}.  Then for a corresponding random graph model, we use the following S. Janson's formula \cite{Janson95} 
\begin{equation}\label{janson}
\lim\limits_{n \rightarrow \infty}  E[L_n] = \lim\limits_{n \rightarrow \infty}  \int\limits_0^1 E[\kappa(G(n,p))]dp -1,
\end{equation}
where $\kappa(G(n,p))$ is the number of components in the Erd\H{o}s-R\'enyi  random graph $G(n,p)$,
and prove the following statement.

\begin{thm}\label{thm:MLmult}
Let $L_n$ denote the length of the the minimal spanning tree in $K_n$, where edge weights are independent and uniform random variables on $[0,1]$. Then
\be
\label{eqn:proposedprob} 
\lim\limits_{n \rightarrow \infty} E[L_n]=\sum\limits_{k=1}^\infty \int\limits_0^\infty \zeta_k(t)dt,
\ee
where $\zeta_k(t)$ are the solutions \eqref{eqn:solutionFlory} of the corresponding system of the modified Smoluchowski coagulation equations \eqref{eqn:Flory}. 
\end{thm}

\medskip
\noindent
Observe that the above equation \eqref{eqn:proposedprob} reproduces the result \eqref{eqn:frieze1} of Frieze \cite{Frieze}.
Indeed, plugging \eqref{eqn:solutionFlory} into \eqref{eqn:proposedprob}, we obtain
\begin{align*}
\lim\limits_{n \rightarrow \infty} E[L_n] &=\sum\limits_{k=1}^\infty \int\limits_0^\infty \zeta_k(t)dt=\sum\limits_{k=1}^\infty \int\limits_0^\infty {k^{k-2} t^{k-1} \over k!} e^{-kt}dt\\
&=\sum\limits_{k=1}^\infty {1 \over k^3} \int\limits_0^\infty {k^k t^{k-1} \over \Gamma(k)} e^{-kt}dt=\sum\limits_{k=1}^\infty {1 \over k^3}=\zeta(3).
\end{align*}

\medskip
\noindent
Theorem \ref{thm:MLmult} will be proved in Sect.~\ref{sec:proofKn}. There, we give a novel proof to this well known result \cite{Frieze}.
The proof utilizes only the modified Smoluchowski coagulation equations \eqref{eqn:Flory} and the weak convergence results that appear in Section \ref{sec:weakconv} of this paper..

\medskip
\noindent
Here is the heuristics behind the proof of Theorem \ref{thm:MLmult} presented in Sect.~\ref{sec:proofKn}.
We already observed that the Marcus-Lushnikov process ${\bf ML}_n(t)$ corresponding to the multiplicative coalescent process that begins with $n$ singletons is equivalent to the cluster size dynamics in the process $G(n, 1-e^{-t/n})$.  Here
\begin{eqnarray*} 
\lim\limits_{n \rightarrow \infty}  E[L_n] & = & \lim\limits_{n \rightarrow \infty}  \int\limits_0^1 E[\kappa(G(n,p))] \, dp -1 
= \lim\limits_{n \rightarrow \infty}  \int\limits_0^\infty {1 \over n} E[\kappa(G(n,1-e^{-t/n}))] e^{-t/n} dt -1 \\
& = & \lim\limits_{n \rightarrow \infty}  \int\limits_0^\infty  \sum\limits_{k=1}^\infty {1 \over n} E[\kappa^{er}(k, n, 1-e^{-t/n})] e^{-t/n} dt -1\\
& = & \lim\limits_{n \rightarrow \infty}  \int\limits_0^\infty  \sum\limits_{k=1}^\infty  {E[\zeta_k^{[n]}(t)] \over n} e^{-t/n} dt -1,
\end{eqnarray*}
where $\kappa^{er}(k, n, p)$ is the number of components of size $k$ in $G(n,p)$  and $p=1-e^{-t/n}$. 
Therefore, one could informally calculate the limit as follows:  
\begin{eqnarray}\label{eqn:janson} 
\lim\limits_{n \rightarrow \infty}  E[L_n] & = & \sum\limits_{k=1}^\infty \int\limits_0^\infty  \zeta_k(t)dt+\lim\limits_{n \rightarrow \infty} \int\limits_{T_{gel}}^\infty {1 \over n} e^{-t/n} dt -1 \nonumber \\
& = & \sum\limits_{k=1}^\infty \int\limits_0^\infty  \zeta_k(t)dt+\lim\limits_{n \rightarrow \infty} e^{-T_{gel}/n} -1 ~~=  \sum\limits_{k=1}^\infty \int\limits_0^\infty  \zeta_k(t)dt.
\end{eqnarray}
Here $\int\limits_{T_{gel}}^\infty {1 \over n} e^{-t/n} dt$ represents the emergence of one giant component at time $T_{gel}=1$. 


\subsection{Proof of Theorem \ref{thm:MLmult}}\label{sec:proofKn} 

Here we give a rigorous proof using the idea behind the approach in formula \eqref{eqn:janson}. Note that unlike the original proof in Frieze \cite{Frieze}, our proof will not rely on knowing the distribution of sizes and the geometry of clusters in the Erd\H{o}s-R\'enyi process as provided in \cite{ER60}. Nor will it require knowing anything about large clusters or the emergence of a unique giant component at time $T_{gel}=1$. All that we use is the weak convergence results of Kurtz \cite{EK,Kurtz81} that we applied to the Marcus-Lushnikov processes in Section \ref{sec:weakconv}.

\begin{proof}
Observe that 
\begin{equation}\label{eqn:Tlim}
\lim\limits_{t \rightarrow \infty}  \sum\limits_{k=1}^\infty k\zeta_k(t)=\lim\limits_{t \rightarrow \infty}  \sum\limits_{k=1}^\infty  {k^{k-1} t^{k-1} \over k!} e^{-kt}=\lim\limits_{t \rightarrow \infty }{x(t) \over t}  =0.
\end{equation}
Thus, for any given $\epsilon \in (0,1/4)$, we can fix $T \gg T_{gel}$ so large that
\begin{equation}\label{eqn:Tepsilon}
\sum\limits_{k=1}^\infty k\zeta_k(T)  \leq   {\epsilon \over 2}.
\end{equation}
Notice that the above inequality \eqref{eqn:Tepsilon} ties $T$ to $\epsilon$.

\bigskip
\noindent 
Fix integer $K>0$. By the equation \eqref{multAS2} proved in Sect.~\ref{sec:multWL} we know that 
$$\lim\limits_{n \to \infty} \sup\limits_{s \in [0,T]} \left| \sum\limits_{k=1}^K n^{-1}\zeta_k^{[n]}(s)-\sum\limits_{k=1}^K \zeta_k(s)\right|=0 \qquad \text{ a.s.}$$ 
Thus, the probability of the complement of the event
\begin{equation}\label{eqn:eventQ}
Q^\epsilon_{K,T,n}:=\left\{ \sum\limits_{k=1}^K  {k \zeta_k^{[n]}(T) \over n}  \leq \epsilon \right\}
\end{equation}
is decreasing to zero as $n \rightarrow \infty$. Moreover, 
$$q^\epsilon_{K,T}(n):=P(\overline{Q^\epsilon_{K,T,n}})=O(n^{-2})$$
by Proposition \ref{prop:on} in Sect.~\ref{sec:CLT} as $\sum\limits_{k=1}^K  {k \zeta_k^{[n]}(0) \over n}-\sum\limits_{k=1}^K k\zeta_k(0)=0$.

\medskip
\noindent 
We will split $\int\limits_0^\infty  \sum\limits_{k=1}^\infty  {E[\zeta_k^{[n]}(t)] \over n} e^{-t/n} dt$ as follows.
\begin{align}\label{4int}
\int\limits_0^\infty  \sum\limits_{k=1}^\infty  {E[\zeta_k^{[n]}(t)] \over n} e^{-t/n} dt &=\int\limits_0^T  \sum\limits_{k=1}^K  {E[\zeta_k^{[n]}(t)] \over n} e^{-t/n} dt 
& \text{\bf (Term I)}\nonumber \\
&+\int\limits_0^T  \sum\limits_{k=K+1}^\infty  {E[\zeta_k^{[n]}(t)] \over n} e^{-t/n} dt & \text{\bf (Term II)}\nonumber \\
&+\big(1-q^\epsilon_{K,T}(n)\big)\int\limits_T^\infty  \sum\limits_{k=1}^K  {E[\zeta_k^{[n]}(t)~|~Q^\epsilon_{K,T,n}] \over n} e^{-t/n} dt & \text{\bf (Term III)}\nonumber \\
&+\big(1-q^\epsilon_{K,T}(n)\big)\int\limits_T^\infty  \sum\limits_{k=K+1}^\infty  {E[\zeta_k^{[n]}(t)~|~Q^\epsilon_{K,T,n}] \over n} e^{-t/n} dt & \text{\bf (Term IV)}\nonumber \\
&+q^\epsilon_{K,T}(n)\int\limits_T^\infty  \sum\limits_{k=1}^\infty  {E[\zeta_k^{[n]}(t)~|~\overline{Q^\epsilon_{K,T,n}}] \over n} e^{-t/n} dt & \text{\bf (Term V)}
\end{align}
Next, we estimate the terms {\bf I-V} in \eqref{4int}.

\medskip
\noindent 
{\bf Term I.} As it is proven in \eqref{multAS} of Section \ref{sec:weakconv}, 
$~\lim\limits_{n \to \infty} \sup\limits_{s \in [0,T]} \left|n^{-1}\zeta_k^{[n]}(s)-\zeta_k(s)\right|=0~~a.s.~$ on $[0,T]$ for all $k=1,2,\hdots,K$. Therefore,
$$\lim\limits_{n \rightarrow \infty} \int\limits_0^T  \sum\limits_{k=1}^K  {E[\zeta_k^{[n]}(t)] \over n} e^{-t/n} dt =\sum\limits_{k=1}^K \int\limits_0^T  \zeta_k(t)dt.$$

\medskip
\noindent 
{\bf Term II.} Observe that,
$$ \sum\limits_{k=K+1}^\infty  {\zeta_k^{[n]}(t) \over n} \leq {1 \over Kn}\sum\limits_{k=K+1}^\infty k\zeta_k^{[n]}(t) ={1 \over K}\left(1-\sum\limits_{k=1}^K {k\zeta_k^{[n]}(t) \over n}\right)
\leq {1 \over K}.$$
Thus,
$$\int\limits_0^T  \sum\limits_{k=K+1}^\infty  {E[\zeta_k^{[n]}(t)] \over n} e^{-t/n} dt=O\left({T \over K}\right)$$
regardless of the value of $n>0$.

\medskip
\noindent 
{\bf Term III.} 
 Recall that in the theory of Marcus-Lushnikov processes the {\it gel} is the set of all ``large" clusters. By analogy, we define the {\it K-gel} to be the collection of all clusters of mass bigger than $K$. Let $M_{K\,gel}(t)$ denote the total mass of all clusters in the {\it K-gel} at time $t \geq 0$.
 

\medskip
\noindent
Now, conditioning on the event $Q^\epsilon_{K,T,n}$, the mass of the {\it K-gel} is $M_{K\, gel}(t) \geq (1-\epsilon)n$ for all $t \geq T$. Thus each cluster not in {\it K-gel} will be gravitating toward the {\it K-gel} with the rate of at least ${M_{K\,gel}(t) \over n} \geq 1-\epsilon$. 
Consider a cluster that was not in  {\it K-gel} at time $T$. Let $T+L$ be the time it becomes a part of the {\it K-gel}. Then, its contribution to the integral 
$~\int\limits_T^\infty  \sum\limits_{k=1}^K  {E[\zeta_k^{[n]}(t)~|~Q^\epsilon_{K,T,n}] \over n} e^{-t/n} dt~$ is at most
\begin{align*}
\int\limits_T^\infty  {E[{\bf 1}_{[T,T+L]}(t)~|~Q^\epsilon_{K,T,n}] \over n} e^{-t/n} dt \leq & {\int\limits_T^\infty  E[{\bf 1}_{[T,T+L]}(t)~|~Q^\epsilon_{K,T,n}]dt \over n}e^{-T/n}\\
&={E[L~|~Q^\epsilon_{K,T,n}] \over n}e^{-T/n} \leq {1 \over (1-\epsilon)n},
\end{align*}
where
$${\bf 1}_A=\begin{cases}
      1 & \text{ if } t \in A \\
      0 & \text{ if } t \not\in A \\
\end{cases}.$$

\medskip
\noindent
The number of clusters not in {\it K-gel} at time $t \geq T$ is
$$\sum\limits_{k=1}^K  \zeta_k^{[n]}(t) \leq \sum\limits_{k=1}^K  k \zeta_k^{[n]}(t)  \leq \epsilon n.$$
Therefore,
$$\int\limits_T^\infty  \sum\limits_{k=1}^K  {E[\zeta_k^{[n]}(t)~|~Q^\epsilon_{K,T,n}] \over n} e^{-t/n} dt\leq {\epsilon n \over (1-\epsilon)n}={\epsilon \over 1-\epsilon } < 2\epsilon. $$

\medskip
\noindent 
{\bf Term IV.} We let $\mathcal{C}=\{C_1,C_2,C_3,\hdots, C_M \}$ denote the set of all clusters that ever exceeded mass $K$ in the whole history of the process $\big\{{\bf ML}_n(t)\big\}_{t \in [0,\infty)}$. There are less than $n/K$ such clusters, i.e., $M <n/K$. For each $C_i$, the emergence time $a_i$ is the time when a pair of clusters of mass not exceeding $K$ mergers into a new cluster $C_i$ of mass greater than $K$. We enumerate these clusters in the order they emerge. 

\medskip
\noindent
Let $M_i(t)$ denote the mass of cluster $C_i$ at time $t$. Consider a pair of  clusters, $C_i$ and $C_j$, coexisting in the {\it K-gel} at time $t$, each of mass smaller than $n/2$. We split their merger rate into two by saying that $C_i$ absorbs $C_j$ with rate ${1 \over 2n} M_i(t)M_j(t)$, and $C_j$ absorbs $C_i$ with rate ${1 \over 2n} M_i(t)M_j(t)$.  In other words, $C_i$ and $C_j$ merge with rate ${1 \over n} M_i(t)M_j(t)$, and which one of the two clusters absorbs the other is decided with a toss of an independent fair coin.

\medskip
\noindent
There is a finite stopping time 
$$t^*=\min\{t \geq 0 ~:~\exists C_i \in \mathcal{C} ~\text{ with } M_i(t)\geq n/2\}$$ 
when a cluster $C_{i^*}$ has its mass $M_{i^*}(t^*) \geq n/2$. After $t^*$, the rules of interactions of cluster $C_{i^*}$ with the other clusters in $\mathcal{C}$ change as follows. 
For $t > t^*$, $C_{i^*}$ absorbs $C_j$ with rate ${1 \over n} M_{i^*}(t)M_j(t)$, while $C_{i^*}$ itself  cannot be absorbed by any other cluster in $\mathcal{C}$.

\medskip
\noindent
Let $b_i$ denote the time when cluster $C_i$ is absorbed by another cluster in collection $\mathcal{C}$. Naturally, there will be only one survivor $C_{i^*}$ with $b_{i^*}=\infty$. Let $J_i=[a_i,b_i) \cap [T,\infty)$ denote the lifespan of cluster $C_i$. Note that a cluster $C_i$ from the set $\mathcal{C}$ existing at time $t \in [a_i,b_i)$ is absorbed into one of the clusters in the {\it K-gel} with the total instantaneous rate of
$$\lambda_i(t) \geq {1 \over 2n} M_i(t)\big(M_{K\,gel}(t)-M_i(t)\big).$$
Conditioning on the event $Q^\epsilon_{K,T,n}$ defined in \eqref{eqn:eventQ}, we have that if $M_i(t) <n/2$ for $t \in J_i$, then the rate of absorption of $C_i$ into the {\it K-gel} is  
$$\lambda_i(t) \geq {1 \over 2n} M_i(t)\left((1-\epsilon)n -{1 \over 2}n\right) \geq {1 \over 2n} M_i(t)\left({3 \over 4}n -{1 \over 2}n\right)\geq {1 \over 8} M_i(t) > {K \over 8}.$$


\medskip
\noindent
Next, 
\begin{equation}\label{eqn:giant}
\int\limits_T^\infty  \sum\limits_{k=K+1}^\infty  {E[\zeta_k^{[n]}(t)~|~Q^\epsilon_{K,T,n}] \over n} e^{-t/n} dt = \int\limits_T^\infty  {1 \over n} e^{-t/n} dt +\mathcal{E}
\end{equation}
where $\int\limits_T^\infty  {1 \over n} e^{-t/n} dt$ is due to the event $Q^\epsilon_{K,T,n}$ which guarantees the existence of at least one component from $\mathcal{C}$ in the {\it K-gel} for all $t \in [T,\infty)$ and the second term $\mathcal{E}$ is responsible for all the times $t \geq T$ when the number of clusters in the {\it K-gel}  is greater than one. The term $\mathcal{E}$ is bounded as follows 
$$\mathcal{E} \leq \int\limits_T^\infty  {E\left[\sum\limits_{i:~i\not= i^*} {\bf 1}_{J_i}(t)~\big|~Q^\epsilon_{K,T,n} \right] \over n} e^{-t/n} dt.$$
Now, each cluster $C_i$ is gravitating towards the rest of the {\it K-gel} with the rate of at least $K/8$. 
Thus, for each $i\not=i^*$,
$$\int\limits_T^\infty  {E\left[{\bf 1}_{J_i}(t)~|~Q^\epsilon_{K,T,n} \right] \over n} e^{-t/n} dt \leq {E[|J_i| ~|~Q^\epsilon_{K,T,n}]\over n} e^{-{T \over n}} \leq {8 \over nK}.$$
Hence, since the cardinality of set $\mathcal{C}$ is $M < n/K$,
$$\mathcal{E} < {n \over K} \cdot {8 \over nK}={8 \over K^2},$$
and from \eqref{eqn:giant}, we obtain
$$\int\limits_T^\infty  \sum\limits_{k=K+1}^\infty  {E[\zeta_k^{[n]}(t)~|~Q^\epsilon_{K,T,n}] \over n} e^{-t/n} dt=1+O(K^{-2})+O\left({T \over n}\right) ~\text{ as }n \rightarrow \infty,$$
where the term $O(K^{-2})$ does not depend on the value of $n>0$.

\medskip
\noindent 
{\bf Term V.} Here
$$q^\epsilon_{K,T}(n)\int\limits_T^\infty  \sum\limits_{k=1}^\infty  {E[\zeta_k^{[n]}(t)~|~\overline{Q^\epsilon_{K,T,n}}] \over n} e^{-t/n} dt ~\leq nq^\epsilon_{K,T}(n) \int\limits_T^\infty   {1 \over n} e^{-t/n} dt ~\leq nq^\epsilon_{K,T}(n)=O(n^{-1}).$$

\bigskip
\noindent
Finally, by putting together the analysis in {\bf Terms I-V} in the equation \eqref{4int}, we obtain for a given fixed $\epsilon \in (0,1/4)$, sufficiently large fixed $T \gg T_{gel}$ satisfying \eqref{eqn:Tepsilon}, and arbitrarily large $K$,
\begin{equation}\label{eqn:rigorKn}
\int\limits_0^\infty  \sum\limits_{k=1}^\infty  {E[\zeta_k^{[n]}(t)] \over n} e^{-t/n} dt=\sum\limits_{k=1}^K \int\limits_0^T  \zeta_k(t)dt+1+O\left({T \over K}\right)+O(K^{-2})+O(\epsilon)+O\left({T \over n}\right)+O(n^{-1}),
\end{equation}
which, when we increase $n$  to infinity will yield
\begin{align*}
\limsup\limits_{n \rightarrow \infty}\left|\int\limits_0^\infty  \sum\limits_{k=1}^\infty  {E[\zeta_k^{[n]}(t)] \over n} e^{-t/n} dt-\sum\limits_{k=1}^\infty \int\limits_0^\infty  \zeta_k(t)dt-1\right|
=&\sum\limits_{k=K+1}^\infty \int\limits_0^T  \zeta_k(t)dt+\sum\limits_{k=1}^\infty \int\limits_T^\infty  \zeta_k(t)dt\\
&+O\left({T \over K}\right)+O(K^{-2})+O(\epsilon).
\end{align*}
Consequently, taking $\limsup\limits_{K \rightarrow \infty}$, we obtain
$$\limsup\limits_{n \rightarrow \infty}\left|\int\limits_0^\infty  \sum\limits_{k=1}^\infty  {E[\zeta_k^{[n]}(t)] \over n} e^{-t/n} dt-\sum\limits_{k=1}^\infty \int\limits_0^\infty  \zeta_k(t)dt-1\right|
=\sum\limits_{k=1}^\infty \int\limits_T^\infty  \zeta_k(t)dt+O(\epsilon).$$
Finally, formula \eqref{eqn:Tepsilon} guarantees that decreasing $\epsilon$ down to zero will propel $T$ to $+\infty$, and
$$\lim\limits_{n \rightarrow \infty}   \int\limits_0^\infty  \sum\limits_{k=1}^\infty  {E[\zeta_k^{[n]}(t)] \over n} e^{-t/n} dt=\sum\limits_{k=1}^\infty \int\limits_0^\infty  \zeta_k(t)dt+1.$$
Thus we confirmed formula \eqref{eqn:proposedprob} for the case of the multiplicative coalescent process.
\end{proof}

\subsection{Relation of Erd\H{o}s-R\'enyi process on $K_{\alpha[n], \beta[n]}$ to cross-multiplicative coalescent}\label{sec:ERonKnn}
Let $\alpha, \beta >0$ be given, and consider two integer valued functions, $\alpha[n]=\alpha n +o(\sqrt{n})$ and $\beta[n]=\beta n +o(\sqrt{n})$.
Next, we introduce the Erd\H{o}s-R\'enyi random graph process on the bipartite graph $K_{\alpha[n],\beta[n]}$ with $\alpha[n]$ vertices on the left side and $\beta[n]$ vertices on the right side. In this random graph process on $K_{\alpha[n],\beta[n]}$, for each edge $e$ of $\alpha[n]\beta[n]=\alpha\beta n^2+o(n\sqrt{n})$ edges we have an associated random variable $U_e$, distributed uniformly on $[0,1]$. The random variables $\{U_e\}_e$ are assumed to be independent. For the ``time" parameter $p \in [0,1]$, an edge $e$ is considered ``open" if $U_e \leq p$. Erd\H{o}s-R\'enyi random graph $G(n,p)$ will consist of all $n$ vertices and all open edges at time $p$.

\medskip
\noindent
In this Erd\H{o}s-R\'enyi random graph process, the probability of two components merging at a given time depends only on the number of edges that connect those two components. If connected component $C_i$ and $C_j$ have partition sizes $(i_1,i_2)$ and $(j_1,j_2)$ respectively, then there are $i_1j_2+i_2j_1$ edges which, when opened, would connect $C_i$ and $C_j$. 
Therefore, the cross-multiplicative coalescent process represents the cluster dynamics of the above Erd\H{o}s-R\'enyi random graph process on the bipartite graph $K_{\alpha[n],\beta[n]}$ under the time change $p=1-e^{-t/n}$. This coalescent process representation is obtained by letting each cluster connecting $i_1$ vertices on the left side of the bipartite graph with $i_2$ vertices on the right side of the bipartite graph be assigned a two-dimensional weight vector $\left[\!\!\begin{array}{c}i_1 \\i_2\end{array}\!\!\right]$. 
Then, the Marcus-Lushnikov process $\zeta^{[n]}_{i_1,i_2}(t)$ corresponding to the cross-multiplicative coalescent process will count the number of clusters with the weight vector $\left[\!\!\begin{array}{c}i_1 \\i_2\end{array}\!\!\right]$ at time $t$.

\subsection{The length of the minimal spanning tree on $K_{\alpha[n], \beta[n]}$ via $\zeta_{i_1,i_2}(t)$.}\label{sec:LnKnn}
Consider the Erd\H{o}s-R\'enyi random graph  model on a complete bipartite graph $K_{\alpha[n], \beta[n]}$.
Let us interpret $U_e$ as the length of edge $e$. Then one can construct a minimal spanning tree on $K_{\alpha[n], \beta[n]}$. Let random variable $L_n$ denote the length of such minimal spanning tree. We want to represent the asymptotic limit of the mean value of $L_n$ via $\zeta_{i_1,i_2}(t)$.

For a random graph process $G(n,p)$ over $K_{\alpha[n], \beta[n]}$, Lemma 1 in Beveridge et al \cite{BFMcD} implies
\begin{equation}\label{jansonKnn}
E[L_n] =  \int\limits_0^1 E[\kappa(G(n,p))] dp -1,
\end{equation}
where $\kappa(G(n,p))$ is the number of components in the random graph process $G(n,p)$ at time $p$.
This will be used in Sect.~\ref{sec:proofKnn} for proving the following theorem.

\begin{thm}\label{thm:mainKnn}
Let $\alpha, \beta >0$ and $L_n=L_n(\alpha,\beta)$ be the length of a minimal spanning tree on a complete bipartite graph $K_{\alpha[n], \beta[n]}$ with partitions of sizes $$\alpha[n]=\alpha n +o(\sqrt{n}) ~\text{ and }~\beta[n]=\beta n +o(\sqrt{n})$$ 
and independent uniform edge weights over $[0, 1]$. Then
\begin{equation}\label{eqn:mainKnn}
\lim\limits_{n \rightarrow \infty} E[L_n]=\sum\limits_{i_1,i_2}^\infty \int\limits_0^\infty \zeta_{i_1,i_2}(t)dt.
\end{equation}
where $\zeta_{i_1,i_2}(t)$ indexed by $\mathbb{Z}_+^2 \setminus \{(0,0)\}$ is the solution of the modified Smoluchowski coagulation system \eqref{eqn:Knnt} with the initial conditions $\zeta_{i_1,i_2}(0)=\alpha \delta_{1,i_1}\delta_{0,i_2}+\beta \delta_{0,i_1}\delta_{1,i_2}$.
\end{thm}

Observe that if we plug-in the solutions \eqref{eqn:sol} of the reduced system of Smoluchowski coagulation equations \eqref{eqn:Knnt} into the right hand side of \eqref{eqn:mainKnn}, we get
\begin{align}\label{eqn:LnKnn}
\sum\limits_{i_1,i_2}^\infty \int\limits_0^\infty \zeta_{i_1,i_2}(t)dt &={\alpha \over \beta}+{\beta \over \alpha}+\sum\limits_{i_1 \geq 1:~i_2 \geq 1}\alpha^{i_1}\beta^{i_2}S_{i_1,i_2}\int\limits_0^\infty t^{i_1+i_2-1} e^{-(\beta i_1+\alpha i_2)t} dt \nonumber \\
& ={\alpha \over \beta}+{\beta \over \alpha}+\sum\limits_{i_1 \geq 1:~i_2 \geq 1}{\alpha^{i_1}\beta^{i_2}S_{i_1,i_2} \over (\beta i_1+\alpha i_2)^{i_1+i_2}}(i_1+i_2-1)! \nonumber \\
& =\gamma+{1 \over \gamma}+\sum\limits_{i_1 \geq 1:~i_2 \geq 1}{\gamma^{i_1}S_{i_1,i_2} \over (i_1+\gamma i_2)^{i_1+i_2}}(i_1+i_2-1)!
\end{align}
with $\gamma={\alpha \over \beta}$.

\medskip
\noindent
Next, by combining Lemma \ref{lem:S} with \eqref{eqn:LnKnn} we obtained the following important theorem.

\begin{thm}\label{thm:main}
Let $\alpha, \beta >0$, $\gamma=\alpha/\beta$, and $L_n=L_n(\alpha,\beta)$ be the length of a minimal spanning tree on a complete bipartite graph $K_{\alpha[n], \beta[n]}$ with partitions of sizes $$\alpha[n]=\alpha n +o(\sqrt{n}) ~\text{ and }~\beta[n]=\beta n +o(\sqrt{n})$$
and independent uniform edge weights over $[0, 1]$. Then the limiting mean length of the minimal spanning tree is
$$\lim\limits_{n \rightarrow \infty} E[L_n] =\gamma+{1 \over \gamma}+\sum\limits_{i_1 \geq 1;~i_2 \geq 1}{(i_1+i_2-1)!\over i_1!i_2!}{\gamma^{i_1}i_1^{i_2-1}i_2^{i_1-1}\over (i_1+\gamma i_2)^{i_1+i_2}}.$$
\end{thm}

\medskip
\noindent
Theorem \ref{thm:main} is consistent with \cite{FMcD}, where it was shown that for $\alpha=\beta$, $~\lim\limits_{n \rightarrow \infty} E[L_n]=2\zeta(3)$.
Indeed, we have the following Corollary reproducing the results in \cite{FMcD}. Observe however that for $\alpha \not= \beta$ the bipartite graph is irregular and the results in Frieze and McDiarmid \cite{FMcD} no longer apply.

\begin{cor}\label{cor:FMcD}
If $\gamma=1$, then
$$\lim\limits_{n \rightarrow \infty} E[L_n]=2\zeta(3).$$
\end{cor}
\begin{proof}
Abel's binomial theorem \cite{Comtet,Riordan} states that
$$\sum\limits_{k=0}^n \binom{n}{k} (x-kz)^{k-1}(y+kz)^{n-k}=x^{-1}(x+y)^n.$$
Plugging-in $x=nz \not=0$, $y=0$, and $i=n-k$, we obtain
$$\sum\limits_{i=0}^n \binom{n}{i} i^{n-i-1}(n-i)^i=n^{n-1}$$
and therefore,
$$\sum\limits_{i_1,i_2:~i_1+i_2=n}i_1 S_{i_1,i_2}=\sum\limits_{i_1,i_2:~i_1+i_2=n}i_1 {i_1^{i_2-1}i_2^{i_1-1} \over i_1!i_2!}={n^{n-1} \over n!}.$$
Hence, 
$$n \cdot \!\!\!\!\! \sum\limits_{i_1,i_2:~i_1+i_2=n}S_{i_1,i_2}=\sum\limits_{i_1,i_2:~i_1+i_2=n}(i_1+i_2) S_{i_1,i_2}=2\sum\limits_{i_1,i_2:~i_1+i_2=n}i_1 S_{i_1,i_2}=2{n^{n-1} \over n!}$$
and
$$\sum\limits_{i_1,i_2:~i_1+i_2=n}S_{i_1,i_2}=2{n^{n-2} \over n!}.$$
Plugging the above into \eqref{eqn:LnKnn} with $\gamma=1$, we obtain
\begin{align}\label{eqn:LnKnnS}
\lim\limits_{n \rightarrow \infty} E[L_n] &=2+\sum\limits_{i_1 \geq 1:~i_2 \geq 1}{S_{i_1,i_2} \over (i_1+i_2)^{i_1+i_2}}(i_1+i_2-1)! \nonumber \\
& =2+\sum\limits_{n=2}^\infty \left( \sum\limits_{i_1, i_2: ~i_1+i_2=n}{S_{i_1,i_2} \over n^n}\right)(n-1)!\nonumber \\
& =2+\sum\limits_{n=2}^\infty 2{n^{n-2} \over n!}\cdot{1 \over n^n}(n-1)!  \nonumber \\
& =2+\sum\limits_{n=2}^\infty {2 \over n^3}=2\zeta(3).
\end{align}
Thus confirming the results in \cite{FMcD}.
\end{proof}

\subsection{Proof of Theorem \ref{thm:mainKnn}}\label{sec:proofKnn}

Let us give a rigorous proof of Theorem \ref{thm:mainKnn}. Here, we will follow the strategy used for proving Theorem \ref{thm:MLmult} in Sect.~\ref{sec:proofKn}.

\begin{proof}
Observe that 
\begin{equation}\label{eqn:TlimCross}
\lim\limits_{t \rightarrow \infty}  \sum\limits_{i_1,i_2} i_1\zeta_{i_1,i_2}(t)=0 \quad \text { and } \quad \lim\limits_{t \rightarrow \infty}\sum\limits_{i_1,i_2} i_2\zeta_{i_1,i_2}(t)=0.
\end{equation}
Indeed, by plugging in $\zeta_{i_1,i_2}(t)$ as in \eqref{eqn:zetasol}, we obtain
\begin{align*}
{d \over dt}\sum\limits_{i_1,i_2} i_1\zeta_{i_1,i_2}(t) &=  \sum\limits_{i_1,i_2} i_1\zeta_{i_1,i_2}(t)\left({i_1+i_2-1 \over t}-(\beta i_1+\alpha i_2)\right)
\leq - {\alpha \wedge \beta \over 2} \sum\limits_{i_1,i_2} i_1\zeta_{i_1,i_2}(t)
\end{align*}
for $t>{1 \over \alpha \wedge \beta}$. Thus, $\sum\limits_{i_1,i_2} i_1\zeta_{i_1,i_2}(t)$, and similarly $\sum\limits_{i_1,i_2} i_2\zeta_{i_1,i_2}(t)$, would decrease to zero
exponentially fast when  $t>{1 \over \alpha \wedge \beta}$.

\bigskip
\noindent
Now, having established \eqref{eqn:TlimCross}, for any given $\epsilon \in (0,1/4)$, we can fix $T \gg T_{gel}$ so large that
\begin{equation}\label{eqn:TepsilonCross}
\sum\limits_{i_1,i_2} i_1\zeta_{i_1,i_2}(t)\leq{\alpha\epsilon \over 2} \quad \text { and } \quad \sum\limits_{i_1,i_2} i_2\zeta_{i_1,i_2}(t)  \leq   {\beta\epsilon \over 2}.
\end{equation}
Notice that the above inequalities \eqref{eqn:TepsilonCross} ties $T$ to $\epsilon$.

\bigskip
\noindent
Fix integers $K_1>0$ and $K_2>0$, and let $R:=R(K_1,K_2)=\{1,2,\hdots,K_1\} \times \{1,2,\hdots,K_2\}$. 
By the equation \eqref{crossAS2} in Sect.~\ref{sec:hydcross} we have
$$\lim\limits_{n \to \infty} \sup\limits_{s \in [0,T]} \left| n^{-1}\sum\limits_R\zeta_{i_1,i_2}^{[n]}(s)-\sum\limits_R \zeta_{i_1,i_2}(s)\right|=0 \qquad \text{ a.s.}$$
Thus, the probability of the complement of the event
\begin{equation}\label{eqn:eventQcross}
Q^\epsilon_{R,T,n}:=\left\{ \sum\limits_{{\bf i} \in R}i_1{\zeta^{[n]}_{i_1,i_2}(T) \over n}\leq {3 \over 4}\alpha\epsilon  \quad \text { and } \quad \sum\limits_{{\bf i} \in R}i_2{\zeta^{[n]}_{i_1,i_2}(T) \over n}\leq {3 \over 4}\beta\epsilon \right\}
\end{equation}
is decreasing to zero as $n \rightarrow \infty$. Moreover, 
$$q^\epsilon_{R,T}(n):=P(\overline{Q^\epsilon_{R,T,n}})=O(n^{-2})$$
by Proposition \ref{prop:on} in Sect.~\ref{sec:CLT} since
$$\lim\limits_{n \rightarrow \infty}\sqrt{n}\left(\sum\limits_{{\bf i} \in R}i_1{\zeta^{[n]}_{i_1,i_2}(0) \over n}-\sum\limits_{{\bf i} \in R}i_1\zeta_{i_1,i_2}(0) \right)=\lim\limits_{n \rightarrow \infty}\sqrt{n}\big(\alpha[n]/n -\alpha\big)=0$$
and
$$\lim\limits_{n \rightarrow \infty}\sqrt{n}\left(\sum\limits_{{\bf i} \in R}i_2{\zeta^{[n]}_{i_1,i_2}(0) \over n}-\sum\limits_{{\bf i} \in R}i_2\zeta_{i_1,i_2}(0) \right)=\lim\limits_{n \rightarrow \infty}\sqrt{n}\big(\beta[n]/n -\beta\big)=0.$$

\medskip
\noindent
We know from \eqref{jansonKnn} that
$$\lim\limits_{n \rightarrow \infty}  E[L_n] = \lim\limits_{n \rightarrow \infty}  \int\limits_0^1 E[\kappa(G(n,p))] dp -1=\lim\limits_{n \rightarrow \infty} \int\limits_0^\infty  \sum\limits_{i_1,i_2}  {E[\zeta^{[n]}_{i_1,i_2}(t)] \over n} e^{-t/n} dt-1$$
provided the latter limit exists.

\medskip
\noindent 
We will split $\int\limits_0^\infty  \sum\limits_{i_1,i_2}  {E[\zeta^{[n]}_{i_1,i_2}(t)] \over n} e^{-t/n} dt$ as follows.
\begin{align}\label{4intCross}
\int\limits_0^\infty  \sum\limits_{i_1,i_2}  {E[\zeta^{[n]}_{i_1,i_2}(t)] \over n} e^{-t/n} dt &=\int\limits_0^T  \sum\limits_{{\bf i} \in R}  {E[\zeta^{[n]}_{i_1,i_2}(t)] \over n} e^{-t/n} dt 
& \text{\bf (Term I)}\nonumber \\
&+\int\limits_0^T  \sum\limits_{{\bf i} \not\in R} {E[\zeta^{[n]}_{i_1,i_2}(t)] \over n} e^{-t/n} dt & \text{\bf (Term II)}\nonumber \\
&+\big(1-q^\epsilon_{R,T}(n)\big)\int\limits_T^\infty  \sum\limits_{{\bf i} \in R}  {E[\zeta^{[n]}_{i_1,i_2}(t)~|~Q^\epsilon_{R,T,n}] \over n} e^{-t/n} dt & \text{\bf (Term III)}\nonumber \\
&+\big(1-q^\epsilon_{R,T}(n)\big)\int\limits_T^\infty  \sum\limits_{{\bf i} \not\in R}  {E[\zeta^{[n]}_{i_1,i_2}(t)~|~Q^\epsilon_{R,T,n}] \over n} e^{-t/n} dt & \text{\bf (Term IV)}\nonumber \\
&+q^\epsilon_{R,T}(n)\int\limits_T^\infty  \sum\limits_{i_1,i_2}  {E[\zeta^{[n]}_{i_1,i_2}(t)~|~\overline{Q^\epsilon_{R,T,n}}] \over n} e^{-t/n} dt & \text{\bf (Term V)}
\end{align}
Next, we estimate the terms {\bf I-V} in \eqref{4intCross}.

\medskip
\noindent 
{\bf Term I.} As we establish in \eqref{crossAS} of Section \ref{sec:weakconv}, 
$~\lim\limits_{n \to \infty}\!\sup\limits_{s \in [0,T]} \!\left|n^{-1}\zeta_{i_1,i_2}^{[n]}(s)-\zeta_{i_1,i_2}(s)\right| = 0 ~~a.s.$ 
on $[0,T]$ for all ${\bf i}=\left[\!\!\begin{array}{c}i_1 \\i_2\end{array}\!\!\right] \in R$. Therefore,
$$\lim\limits_{n \rightarrow \infty} \int\limits_0^T  \sum\limits_{{\bf i} \in R}  {E[\zeta^{[n]}_{i_1,i_2}(t)] \over n} e^{-t/n} dt  = \sum\limits_{{\bf i} \in R} \int\limits_0^T  \zeta_{i_1,i_2}(t) dt.$$

\medskip
\noindent 
{\bf Term II.} Observe that,
\begin{align*}
 \sum\limits_{{\bf i} \not\in R} {\zeta^{[n]}_{i_1,i_2}(t) \over n} &\leq {1 \over n}\sum\limits_{i_1 >K_1}\sum\limits_{i_2} \zeta^{[n]}_{i_1,i_2}(t)+{1 \over n}\sum\limits_{i_1}\sum\limits_{i_2 >K_2} \zeta^{[n]}_{i_1,i_2}(t)\\
& \leq {1 \over K_1n}\sum\limits_{i_1 >K_1}\sum\limits_{i_2} i_1\zeta^{[n]}_{i_1,i_2}(t)+{1 \over nK_2}\sum\limits_{i_1}\sum\limits_{i_2 >K_2} i_2\zeta^{[n]}_{i_1,i_2}(t)\\
 & \leq  {\alpha[n] \over K_1n}+{\beta[n] \over nK_2} \leq 2{\alpha \over K_1}+2{\beta \over K_2}\\
 \end{align*}
for all $n$ large enough. Thus,
$$\int\limits_0^T  \sum\limits_{{\bf i} \not\in R} {E[\zeta^{[n]}_{i_1,i_2}(t)] \over n} e^{-t/n} dt=O\left({T \over K_1}\right)+O\left({T \over K_2}\right).$$

\medskip
\noindent 
{\bf Term III.} 
We define the {\it R-gel} to be the collection of all clusters whose mass vector is not in $R$. Let 
\begin{equation}\label{eqn:MRgel}
M_{R\,gel}(t)=\left[\!\!\begin{array}{c}m_1(t) \\m_2(t)\end{array}\!\!\right]
\end{equation} 
denote the total mass vector of all clusters in the {\it R-gel} at time $t \geq 0$.

\medskip
\noindent
Now, conditioning on the event $Q^\epsilon_{R,T,n}$, we have $m_1(t) \geq \alpha(1-\epsilon)n$ and $m_2(t) \geq \beta(1-\epsilon)n$ for all $t \geq T$, and $n$ large enough. Thus each cluster in $R$ will be gravitating toward the {\it R-gel} with the rate of at least $(\alpha \wedge \beta)(1-\epsilon)$. 
Consider a cluster in $R$ at time $T$. Let $T+L$ be the time it becomes a part of the {\it R-gel}. Then, its contribution to the integral 
$~\int\limits_T^\infty  \sum\limits_{{\bf i} \in R}  {E[\zeta^{[n]}_{i_1,i_2}(t)~|~Q^\epsilon_{R,T,n}] \over n} e^{-t/n} dt~$ is at most
$$\int\limits_T^\infty  {E[{\bf 1}_{[T,T+L]}(t)~|~Q^\epsilon_{R,T,n}] \over n} e^{-t/n} dt \leq {E[L~|~Q^\epsilon_{R,T,n}] \over n}e^{-T/n} \leq {1 \over (\alpha \wedge \beta)(1-\epsilon)n}.$$

\medskip
\noindent
The number of clusters in $R$ at time $t \geq T$ is
$$ \sum\limits_{{\bf i} \in R}  \zeta^{[n]}_{i_1,i_2}(t) \leq \sum\limits_{{\bf i} \in R}  (i_1+i_2)\zeta^{[n]}_{i_1,i_2}(t)  \leq (\alpha+\beta)\epsilon n.$$
Therefore,
$$\int\limits_T^\infty  \sum\limits_{{\bf i} \in R}  {E[\zeta^{[n]}_{i_1,i_2}(t)~|~Q^\epsilon_{R,T,n}] \over n} e^{-t/n} dt \leq {(\alpha+\beta)\epsilon n \over (\alpha \wedge \beta)(1-\epsilon)n}={2\epsilon \over 1-\epsilon } < 3\epsilon. $$

\medskip
\noindent 
{\bf Term IV.} We let $\mathcal{C}=\{C_1,C_2,C_3,\hdots, C_M \}$ denote the set of all clusters whose mass vectors ever exceeded $K_1$ in the first coordinate and/or  ever exceeded $K_2$ in the second coordinate in the history of the process ${\bf ML}_n(t)$, i.e., all clusters that were ever a part of {\it R-gel}. The number of clusters in $\mathcal{C}$ is less than $\alpha[n]/K_1+\beta[n]/K_2$. For each $C_i$, the emergence time $a_i$ is the time of a merger of a pair of clusters in $R$, resulting in appearance of a new cluster $C_i$ in {\it R-gel}. We enumerate these clusters in the order they emerge. 

\medskip
\noindent
Let $M_i(t)=\left[\!\!\begin{array}{c}m_{1,i}(t) \\m_{2,i}(t)\end{array}\!\!\right]$ denote the mass vector of cluster $C_i$ at time $t$. Consider a pair of  clusters, $C_i$ and $C_j$, coexisting in the {\it R-gel} at time $t$, such that $m_{1,i}, m_{1,j} < \alpha n/2$ and $m_{2,i}, m_{2,j} < \beta n/2$. We split their merger rate into two by saying that $C_i$ absorbs $C_j$ with rate ${1 \over 2n} \big(m_{1,i}(t)m_{2,j}(t)+m_{2,i}(t)m_{1,j}(t)\big)$, and $C_j$ absorbs $C_i$ with rate ${1 \over 2n} \big(m_{1,i}(t)m_{2,j}(t)+m_{2,i}(t)m_{1,j}(t)\big)$.  

\medskip
\noindent
There is a finite stopping time 
$$t^*=\min\{t \geq 0 ~:~\exists C_i \in \mathcal{C} ~\text{ with } m_{1,i}(t) \geq \alpha n/2~\text{ or }~ m_{2,i}(t) \geq \beta n/2 \}$$ 
when a cluster $C_{i^*}$ has its mass vector satisfying either $m_{1,i^*}(t^*) \geq \alpha n/2$ or $m_{2,i^*}(t^*) \geq \beta n/2$. 
After time $t^*$ the rules of interactions of cluster $C_{i^*}$ with the other clusters in $\mathcal{C}$ change as follows. For $t > t^*$,
$C_{i^*}$ absorbs $C_j$ with rate ${1 \over n} \big(m_{1,i^*}(t)m_{2,j}(t)+m_{2,i^*}(t)m_{1,j}(t)\big)$, while $C_{i^*}$ itself cannot be absorbed by any other cluster in $\mathcal{C}$.

\medskip
\noindent
Let $b_i$ denote the time when cluster $C_i$ is absorbed by another cluster in collection $\mathcal{C}$. Naturally, there will be only one survivor $C_{i^*}$ with $b_{i^*}=\infty$. Let $J_i=[a_i,b_i) \cap [T,\infty)$ denote the lifespan of cluster $C_i$. Note that a cluster $C_i$ from the collection $\mathcal{C}$ existing at time $t \in [a_i,b_i)$ is absorbed into one of the clusters in the {\it R-gel} with the total instantaneous rate of
$$\lambda_i(t) \geq {1 \over 2n} \Big(m_{1,i}(t)\big(m_2(t) -m_{2,i}(t)\big)+m_{2,i}(t)\big(m_1(t)-m_{1,i}(t) \big)\Big),$$
where $m_1(t)$ and $m_2(t)$ are as defined in \eqref{eqn:MRgel}.
Conditioning on the event $Q^\epsilon_{R,T,n}$ defined in \eqref{eqn:eventQcross}, we have that if $m_{1,i}(t) < \alpha n/2$ and  $m_{2,i}(t) < \beta n/2$ for $t \in J_i$, then the rate of absorption of $C_i$ into the {\it R-gel} is  
\begin{align*}
\lambda_i(t) &\geq {1 \over 2n} m_{1,i}(t)\beta\left((1-\epsilon)n -{1 \over 2}n\right)+{1 \over 2n} m_{2,i}(t)\alpha\left((1-\epsilon)n -{1 \over 2}n\right)\\ 
&\geq {1 \over 2n} m_{1,i}(t)\beta \left({3 \over 4}n -{1 \over 2}n\right)+{1 \over 2n} m_{2,i}(t)\alpha \left({3 \over 4}n -{1 \over 2}n\right)\\
&\geq {m_{1,i}(t)\beta+m_{2,i}(t)\alpha \over 8}  > {K_1\beta+K_2\alpha \over 8}.
\end{align*}

\medskip
\noindent
Next, 
\begin{equation}\label{eqn:giantCross}
\int\limits_T^\infty  \sum\limits_{{\bf i} \not\in R}  {E[\zeta^{[n]}_{i_1,i_2}(t)~|~Q^\epsilon_{R,T,n}] \over n} e^{-t/n} dt  = \int\limits_T^\infty  {1 \over n} e^{-t/n} dt +\mathcal{E}
\end{equation}
where $\int\limits_T^\infty  {1 \over n} e^{-t/n} dt$ is due to the event $Q^\epsilon_{R,T,n}$ which guarantees the existence of at least one component from $\mathcal{C}$ in the {\it R-gel} for all $t \in [T,\infty)$ and the second term $\mathcal{E}$ is responsible for all the times $t \geq T$ when the number of clusters in the {\it R-gel}  is greater than one. The term $\mathcal{E}$ is bounded as follows 
$$\mathcal{E} \leq \int\limits_T^\infty  {E\left[\sum\limits_{i:~i\not= i^*} {\bf 1}_{J_i}(t)~\big|~Q^\epsilon_{R,T,n} \right] \over n} e^{-t/n} dt.$$
Now, each cluster $C_i$ is gravitating towards the rest of the {\it R-gel} with the rate of at least ${K_1\beta+K_2\alpha \over 8}$. 
Thus, for each $i\not=i^*$,
$$\int\limits_T^\infty  {E\left[{\bf 1}_{J_i}(t)~|~Q^\epsilon_{R,T,n} \right] \over n} e^{-t/n} dt \leq {E[|J_i| ~|~Q^\epsilon_{R,T,n}]\over n} e^{-{T \over n}} \leq {8 \over n(K_1\beta+K_2\alpha)}.$$
Hence, since the cardinality of set $\mathcal{C}$ is $M < \alpha[n]/K_1+\beta[n]/K_2$,
$$\mathcal{E} < (\alpha[n]/K_1+\beta[n]/K_2) \cdot {8 \over n(K_1\beta+K_2\alpha)}={8(\alpha/K_1+\beta/K_2) \over K_1\beta+K_2\alpha}+o(1),$$
and from \eqref{eqn:giantCross}, we obtain
$$\int\limits_T^\infty  \sum\limits_{{\bf i} \not\in R}  {E[\zeta^{[n]}_{i_1,i_2}(t)~|~Q^\epsilon_{R,T,n}] \over n} e^{-t/n} dt =1+O(K_1^{-2})+O(K_2^{-2})+O\left({T \over n}\right)+o(1) ~\text{ as }n \rightarrow \infty.$$

\medskip
\noindent 
{\bf Term V.} Here
\begin{align*}
q^\epsilon_{R,T}(n)\int\limits_T^\infty  \sum\limits_{i_1,i_2}  {E[\zeta^{[n]}_{i_1,i_2}(t)~|~\overline{Q^\epsilon_{R,T,n}}] \over n} e^{-t/n} dt  &\leq q^\epsilon_{R,T}(n) \int\limits_T^\infty   {\alpha[n]+\beta[n] \over n} e^{-t/n} dt\\
&\leq (\alpha[n]+\beta[n])q^\epsilon_{R,T}(n) ~=O(n^{-1})
\end{align*}
as $q^\epsilon_{R,T}(n)=O(n^{-2})$.

\bigskip
\noindent
Finally, by putting together the analysis in {\bf Terms I-V} in the equation \eqref{4intCross}, we obtain for a given fixed $\epsilon \in (0,1/4)$, sufficiently large fixed $T \gg T_{gel}$ satisfying \eqref{eqn:TepsilonCross}, and arbitrarily large $K_1$ and $K_2$,
\begin{align}\label{eqn:rigorKnCross}
\int\limits_0^\infty  \sum\limits_{i_1,i_2}  {E[\zeta^{[n]}_{i_1,i_2}(t)] \over n} e^{-t/n} dt=&\sum\limits_{{\bf i} \in R(K_1,K_2)} \int\limits_0^T  \zeta_{i_1,i_2}(t) dt+1 +O\left({T \over K_1}\right)+O\left({T \over K_2}\right)\nonumber \\
&+O(K_1^{-2})+O(K_2^{-2})+O(\epsilon)+O\left({T \over n}\right)+O(n^{-1}),
\end{align}
which when we increase $n$  to infinity will yield
$$\lim\limits_{n \rightarrow \infty}   \int\limits_0^\infty  \sum\limits_{i_1,i_2}  {E[\zeta^{[n]}_{i_1,i_2}(t)] \over n} e^{-t/n} dt=\sum\limits_{i_1,i_2}  \int\limits_0^\infty \zeta_{i_1,i_2}(t)dt+1.$$
\end{proof}

\section{Hydrodynamic limits for Marcus-Lushnikov processes}\label{sec:weakconv}

In \cite{Kurtz81} and \cite{EK}, a certain class of Markov processes, called {\it density dependent population process}, was considered. These are jump Markov processes which depend on a certain parameter $n$ which can be interpreted depending on the context of a model. Usually it represents the population size. Many coalescent processes can be restated as a case of density dependent population processes if all cluster weights are integers. There, the total mass $n$ is the parameter representing the population size. Specifically, we may assume that the coalescent process starts with $n$ clusters of unit mass each (aka singletons). In Kurtz \cite{Kurtz81} and in Chapter 11 of Ethier and Kurtz \cite{EK}, the law of large numbers and the central limit theorems were established for such density dependent population processes as $n \rightarrow \infty$. In this section we will adopt these weak limit laws for the multiplicative and cross-multiplicative coalescent processes.

\subsection{Density dependent population processes}
We first formulate the framework for the  convergence result of Kurtz as stated in Theorem 2.1 in Chapter 11 of \cite{EK} (Theorem 8.1 in \cite{Kurtz81}).
There, the {\it density dependent population processes} are defined as continuous time Markov processes with state spaces in $\mathbb{Z}^d$, and transition intensities represented as follows
\begin{equation}\label{tintK}
q^{(n)}(k, k+\ell) = n \left[ \beta_{\ell} \left( \frac{k}{n} \right) + O\left( \frac{1}{n} \right) \right],
\end{equation}
where $\ell, k \in \mathbb{Z}^d$, and $\beta_\ell$ is a given collection of rate functions. 

\medskip
\noindent
In Section 5.1 of \cite{Aldous}, Aldous observes that the results from Chapter 11 of Ethier and Kurtz \cite{EK} can be used to prove the weak convergence of  a Marcus-Lushnikov  process to the solutions of Smoluchowski system of equations in the case when the Marcus-Lushnikov process can be formulated as a finite dimensional density dependent population process. 
Specifically,  the  Marcus-Lushnikov processes  corresponding to the multiplicative and Kingman coalescent with the monodisperse initial conditions ($n$ singletons) can be represented as finite dimensional density dependent population processes defined above.

\medskip
\noindent
Define $F(x)=\sum\limits_\ell \ell \beta_\ell (x)$. Then, Theorem 2.1 in Chapter 11 of \cite{EK} (Theorem 8.1 in \cite{Kurtz81}) states the following law of large numbers. 
Let $\hat{X}_n(t)$ be the Markov process with the intensities $q^{(n)}(k, k+\ell)$ given in \eqref{tintK}, and let $X_n(t)=n^{-1}\hat{X}_n(t)$. Finally, let $|x|=\sqrt{\sum x_i^2 }$ denote the Euclidean norm in $\mathbb{R}^d$.

\begin{thm} \label{kurtzT}
Suppose for all compact $\mathcal{K} \subset \mathbb{R}^d$,
$$\sum_{\ell} |\ell| \sup_{x \in \mathcal{K}} \beta_\ell (\bar x) < \infty,$$
and there exists $M_\mathcal{K}>0$ such that
\begin{equation}\label{KurtzLipschitz}
|F(x)-F(y)| \leq M_\mathcal{K}|x-y|, \qquad \text{ for all }x,y \in \mathcal{K}.
\end{equation}
Suppose $\lim\limits_{n \to \infty}X_n(0)=x_0$, and $X(t)$ satisfies
\begin{equation}\label{KurtzDE}
X(t) = X(0) + \int_0^t F(X(s)) ds,
\end{equation}
for all $T \geq 0$. Then
\begin{equation}\label{KurtzAS}
\lim\limits_{n \to \infty} \sup\limits_{s \in [0,T]} |X_n(s)-X(s)|=0 \qquad \text{ a.s. }
\end{equation}
\end{thm}

\noindent


\subsection{Hydordynamic limit for multiplicative coalescent process}\label{sec:multWL}

Consider a multiplicative coalescent process with kernel $K(i,j)=ij$.
Recall that in the definition of a coalescent process given in Sect.~\ref{sec:mult}, a pair of clusters with masses $i$ and $j$ coalesces at the rate $K(i,j)/n$.
Consider the corresponding Marcus-Lushnikov process 
$${\bf ML}_n(t)=\Big(\zeta_1^{[n]}(t),\zeta_2^{[n]}(t),\hdots, \zeta_n^{[n]}(t), 0,0,\hdots \Big)$$ 
that keeps track for the numbers of clusters in each weight category. 
There, the initial conditions will be ${\bf ML}_n(0)=(n,0,0,\hdots)=ne_1$, where $e_i$ denotes the $i$-th coordinate vector.

\bigskip
\noindent
Next, for a fixed positive integer $K$, let $\hat{X}_n(t)$ be the restriction of process ${\bf ML}_n(t)$ to the first $K$ dimensions, i.e.
$$\hat{X}_n(t)=\Big(\zeta_1^{[n]}(t),\zeta_2^{[n]}(t),\hdots, \zeta_K^{[n]}(t) \Big)$$
with the initial conditions $\hat{X}_n(0)=ne_1$.
Apparently, $\hat{X}_n(t)$ is itself a  (finite dimensional) Markov process with the following transition rates of $\hat{X}_n(t)$ stated as in \eqref{tintK}.  Let $x=(x_1,x_2,\hdots, x_K)$. 
Then, for any pair $1\leq i< j \leq K$, the change vector $\ell=- e_i-e_j+ e_{i+j}{\bf 1}_{i+j \leq K}$ corresponding to a merger of clusters of respective sizes $i$ and $j$ is assigned the rate
$$q^{(n)}(x , x +\ell)={ij \over n}x_i x_j = n \beta_{\ell} \left({x \over n}\right),$$
where $\beta_{\ell}(x)=ijx_ix_j $. 

For a given $1\leq i \leq K$, the change vector $\ell=-2e_i+e_{2i}{\bf 1}_{2i \leq K}$ corresponding to a merger of a pair of clusters of size $i$
is assigned the rate
$$q^{(n)}(x , x +\ell)={1 \over n} \left[\frac{i^2 x_i^2}{2}-\frac{i^2x_i}{2}\right]= n\left[ \beta_{\ell}\left({x \over n}\right)+ O\left( \frac{1}{n} \right) \right],$$
where $\beta_{\ell} (x)=i^2\frac{x_i^2}{2}$.

For a given $1\leq i \leq K$, the change vector $\ell=-e_i$ corresponding to a cluster of mass $i$ merging with a cluster of mass greater than $K$ is assigned the rate
$$q^{(n)}(x , x +\ell)={1 \over n}ix_i \left[n-\sum\limits_{j=1}^K jx_j \right]= n\beta_{\ell}\left({x \over n}\right),$$
where $\beta_{\ell} (x)=ix_i \left(1-\sum\limits_{j=1}^K j x_j \right)$.

\bigskip
\noindent
Then, by Theorem \ref{kurtzT}, $X_n(t)=n^{-1}\hat{X}_n(t)$ converges to $X(t)$ as in \eqref{KurtzAS}, where $X(t)$ satisfies \eqref{KurtzDE} with
\begin{align}\label{Kbl}
F(x) := \sum\limits_\ell \ell \beta_\ell (x)=& \sum_{ij:~1\leq i < j \leq K} ijx_i x_j [- e_i-e_j+ e_{i+j}{\bf 1}_{i+j \leq K}] \nonumber\\
&+{1 \over 2}\sum_{i=1}^K i^2x_i^2 [-2e_i+e_{2i}{\bf 1}_{2i \leq K}] -\sum_{i=1}^K ix_i \left(1-\sum\limits_{j=1}^K jx_j \right) e_i \nonumber \\
&= \sum_{i=1}^K \left(-ix_i+  {1 \over 2}\sum_{\substack{1 \leq i_1,i_2\leq K \\  i_1+i_2=i}}i_1i_2 x_{i_1} x_{i_2}\right)e_i .
\end{align}
Here, $F(x)$ is naturally satisfying the Lipschitz continuity conditions \eqref{KurtzLipschitz}, and the initial conditions $X(0)=X_n(0)=e_1$.

\medskip
\noindent
Observe that the system of equations \eqref{KurtzDE} with $F(x)$ as in (\ref{Kbl}) will yield the reduced system of Smoluckowski coagulation equations \eqref{eqn:Flory} also known as the Flory coagulation system \cite{Flory}. Thus, for a given integer $K>0$ and a fixed real $T>0$,
\begin{equation}\label{multAS}
\lim\limits_{n \to \infty} \sup\limits_{s \in [0,T]} \left|n^{-1}\zeta_k^{[n]}(s)-\zeta_k(s)\right|=0 \qquad \text{ a.s.}
\end{equation}
for $k=1,2,\hdots,K$.

\medskip
\noindent
Note that the above limit no longer requires a fixed $K$ for each individual $k$ in \eqref{multAS}. However, we mainly use the following limit in our calculations,
\begin{equation}\label{multAS2}
\lim\limits_{n \to \infty} \sup\limits_{s \in [0,T]} \left| \sum\limits_{k=1}^K n^{-1}\zeta_k^{[n]}(s)-\sum\limits_{k=1}^K \zeta_k(s)\right|=0 \qquad \text{ a.s.}
\end{equation}

\subsection{Hydordynamic limit for cross-multiplicative coalescent processes}\label{sec:hydcross}

Fix integers $K_1>0$ and $K_2>0$, and let $R:=R(K_1,K_2)=\{1,2,\hdots,K_1\} \times \{1,2,\hdots,K_2\}$. Let $e_{\bf i}$ be the standard basis vectors in $\mathbb{R}^{K_1K_2}$, enumerated by ${\bf i}=\left[\!\!\begin{array}{c}i_1 \\i_2\end{array}\!\!\right] \in R$. Consider a restriction  to  $\left[\!\!\begin{array}{c}i_1 \\i_2\end{array}\!\!\right] \in R$ of a Marcus-Lushnikov process $\zeta_{i_1,i_2}(t)$ with the cross-multiplicative kernel.
Let
$${\hat X}_n(t)=\Big\{\zeta_{i_1,i_2}^{[n]}(t) \Big\}_{{\bf i}\in R}$$
with the initial conditions ${\hat X}_n(0)=\alpha[n]e_{0'}+\beta[n]e_{0''}$, where $0'=\left[\!\!\begin{array}{c}1 \\ 0\end{array}\!\!\right]$ and $0''=\left[\!\!\begin{array}{c}0 \\ 1\end{array}\!\!\right]$.

\bigskip
\noindent
We observe the following transition rates of ${\hat X}_n(t)$ stated as in \eqref{tintK}.  Let $x=\sum\limits_{{\bf i}\in R} x_{\bf i}e_{\bf i}$. 
Then, for any ${\bf i}$ and ${\bf j}$ in $R$, the change vector $\ell=-e_{\bf i} - e_{\bf j} +{\bf 1}_{\{{\bf i+j}\in R\}}e_{\bf i+j}$ corresponding to a merger of clusters of respective weights ${\bf i}$ and ${\bf j}$ is assigned the rate
$$q^{(n)}(x , x +\ell)={1 \over n}(i_1j_2+i_2j_1)x_{\bf i} x_{\bf j} = n \beta_{\ell} (x),$$
where $\beta_{\ell} (x)=(i_1j_2+i_2j_1)x_{\bf i} x_{\bf j} $. 

For a given ${\bf i} \in R$, the change vector $\ell=-e_{\bf i}$ corresponding to the merger of clusters whose weight vector is ${\bf i}$ with clusters whose weight vectors are not in $R$ is assigned  the rate
$$q^{(n)}(x , x +\ell)={1 \over n} \left[i_1x_{\bf i}\left(\beta[n]-\sum_{{\bf j}\in R}  j_2x_{\bf j}\right)+i_2 x_{\bf i}\left(\alpha[n]-\sum_{{\bf j}\in R}  j_1x_{\bf j}\right)\right] = n\left[\beta_{\ell}(x) + O\left( \frac{1}{n} \right) \right] ,$$
where $\beta_{\ell} (x)=i_1x_{\bf i}\left(\beta-\sum_{{\bf j}\in R}  j_2x_{\bf j}\right)+i_2 x_{\bf i}\left(\alpha-\sum_{{\bf j}\in R}  j_1x_{\bf j}\right)$.

\bigskip
\noindent
Thus, by Theorem \ref{kurtzT},  $X_n(t)$ converges to $X(t)$ as in \eqref{KurtzAS}, where $X(t)$ satisfies \eqref{KurtzDE} with
\begin{align}\label{cm}
F(x) := \sum_{\ell} \ell \beta_{\ell} (x) \nonumber = & {1 \over 2}\sum_{{\bf i,j}\in R} \left[ -e_{\bf i} - e_{\bf j} + {\bf 1}_{\{{\bf i+j}\in R\}} e_{\bf i+j}  \right]  (i_1j_2+i_2j_1)x_{\bf i} x_{\bf j} \nonumber \\
& - \sum_{{\bf i}\in R} e_{\bf i} i_1x_{\bf i}\left(\beta-\sum_{{\bf j}\in R}  j_2x_{\bf j}\right) -\sum_{{\bf i}\in R} e_{\bf i} i_2 x_{\bf i}\left(\alpha-\sum_{{\bf j}\in R}  j_1x_{\bf j}\right) \nonumber \\
=&\sum_{{\bf i}\in R} e_{\bf i} \left(-(\beta i_1+ \alpha i_2)x_{\bf i}+ {1 \over 2}\sum\limits_{\bf \ell, k:~ \ell+k=i} (\ell_1 k_2+\ell_2 k_1)x_{\bf \ell} x_{\bf k} \right)
\end{align}
for a fixed $T>0$.
The system of equations \eqref{KurtzDE} with $F(x)$ given in \eqref{cm} will yield the reduced system of Smoluckowski coagulation equations  \eqref{eqn:Knnt}.
So, for a fixed a pair of positive integers $K_1$ and $K_2$, and a fixed real number $T>0$,
\begin{equation}\label{crossAS}
\lim\limits_{n \to \infty} \sup\limits_{s \in [0,T]} \left|n^{-1}\zeta_{i_1,i_2}^{[n]}(s)-\zeta_{i_1,i_2}(s)\right|=0 \qquad \text{ a.s.}
\end{equation}
for all $\left[\!\!\begin{array}{c}i_1 \\i_2\end{array}\!\!\right] \in R$. Consequently,
\begin{equation}\label{crossAS2}
\lim\limits_{n \to \infty} \sup\limits_{s \in [0,T]} \left| n^{-1}\sum\limits_{\substack{1 \leq i_1 \leq K_1\\1 \leq i_2 \leq K_2}}\zeta_{i_1,i_2}^{[n]}(s)-\sum\limits_{\substack{1 \leq i_1 \leq K_1\\1 \leq i_2 \leq K_2}} \zeta_{i_1,i_2}(s)\right|=0 \qquad \text{ a.s.}
\end{equation}

\subsection{Central Limit Theorem and related results}\label{sec:CLT}

The usefulness of the framework set in \cite{EK,Kurtz81}  for proving weak convergence is that the law of large numbers Theorem \ref{kurtzT} is enhanced with the corresponding central limit theorem (see Theorem \ref{KurtzCLT} below) and the large deviation theory \cite{FK}.
The following central limit theorem is derived in Theorem 8.2 in \cite{Kurtz81} (and Theorem 2.3 in Chapter 11 of \cite{EK}). 
\begin{thm}\label{KurtzCLT}
Suppose for all compact $\mathcal{K} \subset \mathbb{R}^d$,
\begin{equation}\label{cond4}
 \sum_{\ell} |\ell|^2 \sup_{x \in \mathcal{K}} \beta_\ell (x) < \infty
\end{equation}
and that the $\beta_\ell$ and $\partial F$ are continuous. Suppose $X_n$ and $X$ are as  in Theorem \ref{kurtzT}, and suppose $V_n=\sqrt{n} (X_n - X)$ is such 
that $\lim_{n\to\infty} V_n(0) = V(0)$, where $V(0)$ is a constant. Then $V_n$ converges in distribution to $V$, which is the solution of 
\begin{equation}\label{VrepLim}
V(t) = V(0) + U(t) +  \int_0^t \partial F(X(s))V(s) ds, 
\end{equation}
where $U(t)$ is a Gaussian process and $\partial F(X(s))=  (\partial_j  F_i(X(s)))_{i,j} $.
\end{thm} 

\medskip
\noindent
The proof of Theorem \ref{KurtzCLT} is based on representing $V_n(t)$ as follows. Let $Y_{\ell}$ be independent Poisson processes with rate one. Then,
\begin{equation}\label{Vrep}
V_n(t) = V_n(0) + U_n(t)+ \int_0^t \sqrt{n} \big(F(X_n(s)) - F(X(s))\big) ds,
\end{equation}
where 
$$U_n(t)=\sum_{\ell} \ell W_\ell ^ {(n)} \Bigl( \int_0^ t \beta_\ell (X_n(s)) ds \Bigr),$$
$W_\ell ^ {(n)}(u) = n^ {-1/2} \hat Y_\ell (nu)$, and $~\hat Y_{\ell} (u) := Y_{\ell}(u) -u~$  are centralized Poisson processes. 

\bigskip
\noindent
Next, we will use formula \eqref{Vrep} in order to derive an upper bound \eqref{eqn:quadratic} on probability $P(|X_n(T)-X(T)| \geq \delta)$. Let us consider a simple case of a density dependent population process on $\mathbb{R}^d$  for which the following three conditions are satisfied.
\begin{description}
  \item[i] $V_n=\sqrt{n} (X_n - X)$ is such that $\lim_{n\to\infty} V_n(0) = V(0)$.
  \item[ii] Both $X_n(t)$ and $X(t)$ live on a compact set $\mathcal{K}$.
  \item[iii] There are finitely many vectors $\ell \in \mathbb{R}^d$ such that $\beta_\ell (x)>0$ for some $x \in \mathcal{K}$.
\end{description}
Notice that the above conditions are satisfied for the Marcus-Lushnikov processes considered here, with the general bilinear kernel as in Sect.~\ref{sec:multWL} and with the cross-multiplicative kernel as in Sect.~\ref{sec:hydcross}. Specifically, for a given $m>0$, let
$$\mathcal{K}_m=\Big\{x \in \mathbb{R}_+^d ~:~ \sum\limits_i x_i \leq m \Big\}.$$
Then, in Sect.~\ref{sec:multWL},  $X_n(t),X(t) \in \mathcal{K}_2$, and in Sect.~\ref{sec:hydcross}, $X_n(t),X(t) \in \mathcal{K}_m$ for $m>\alpha+\beta$.

\bigskip
\noindent
\begin{prop}\label{prop:on}
Assuming the above conditions \textrm{\bf i-iii} are satisfied together with the Lipschitz continuity conditions \eqref{KurtzLipschitz}, we have
\begin{equation}\label{eqn:quadratic}
P(|X_n(T)-X(T)| \geq \delta)=O(n^{-2}).
\end{equation}
\end{prop}

\begin{proof}
Here,
$$\sqrt{n} \big|F(X_n(s)) - F(X(s))\big| \leq \sqrt{n}M_\mathcal{K} |X_n(s)-X(s)|=M_\mathcal{K}|V_n(s)|$$
and for a fixed $T>0$ and any $t \leq T$,
$$|V_n(0) +U_n(t)| \leq \varepsilon_n(T):=|V_n(0)| +\sum_{\ell} |\ell | ~\max\left\{ \left|W_\ell ^ {(n)} (s)  \right| ~:~ s \in \big[0,T\sup_{x \in \mathcal{K}} \beta_\ell (x) \big] \right\}.$$
Hence, for a fixed $T>0$, equation \eqref{Vrep} implies the following inequality,
$$|V_n(t)| \leq \varepsilon_n(T)+ M_\mathcal{K}  \int_0^t |V_n(s)| ds \quad \text{ for all } t\in [0,T].$$
Then, by Gr\"{o}nwall's inequality (see Appendix 5 in \cite{EK}),
\begin{equation}\label{eqn:gronwall}
|V_n(t)| \leq \varepsilon_n(T) e^{M_\mathcal{K} t}.
\end{equation}
In particular, we use equation \eqref{eqn:gronwall} together with Markov inequality to obtain the following simple bound for any $\delta>0$,
\begin{equation}\label{eqn:quadraticV}
P(|X_n(T)-X(T)| \geq \delta) \leq {V_n^4(T) \over n^2 \delta^4} \leq {E[\varepsilon_n^4(T)] e^{4M_\mathcal{K} T} \over n^2 \delta^4}.
\end{equation}
Here, for any fixed real $S>0$, integer $r>0$, and any real $\lambda>0$, we have by Doob's martingale inequality, 
$$P\left( \max\limits_{s \in [0,S]}\left|W_\ell ^ {(n)} (s)  \right|^r \geq \lambda \right) =P\left( \max\limits_{s \in [0,S]}\left|W_\ell ^ {(n)} (s)  \right| \geq \lambda^{1/r} \right)
 \leq {E\left[\left(W_\ell ^ {(n)} (S) \right)^{2+2r} \right] \over  \lambda^{2+2/r}}$$
 as $\left|W_\ell ^ {(n)} (s)  \right|$ is a non-negative sub-martingale. Therefore,
\begin{align*}
E\left[\max\limits_{s \in [0,S]}\left|W_\ell ^ {(n)} (s)  \right|^r \right] & \leq 1+\int\limits_1^\infty P\left( \max\limits_{s \in [0,S]}\left|W_\ell ^ {(n)} (s)  \right|^r \geq \lambda \right)  d\lambda \\
 &\leq 1+(1+2/r)E\left[\left(W_\ell ^ {(n)} (S) \right)^{2+2r} \right],
 \end{align*}
where by the classical central limit theorem,
$$\lim\limits_{n \rightarrow \infty} E\left[\left(W_\ell ^ {(n)} (S) \right)^{2+2r} \right]=S^{1+r} E[Z^{2+2r}] ,\qquad Z \text{ - standard normal r.v.}$$
Thus, $E[\varepsilon_n^4(T)]=O(1)$, and \eqref{eqn:quadratic} follows from \eqref{eqn:quadraticV}.
\end{proof}


\section{Discussion: generalizations and open problems.}\label{sec:dscussion}
In this paper we considered an important example of coagulation ODEs obtained as a hydrodynamic limit of a Marcus-Lushnikov process
that tracks the merger history of a coalescent process with two dimensional weight vectors.   
The coagulation equations and gelation in the Marcus-Lushnikov dynamics for other coalescent processes with multidimensional weight vectors is on its own
an interesting object of studies.  As a natural next step, one may consider a generalization of the existing results \cite{Aldous98,Jeon98,Jeon99,Norris99,EMP02,FG04,FL09} on gelation phenomenon for vector weighted processes. 

An extension of the application to minimal spanning trees may come from an observation that the convergence rates in the hydrodynamic limit yield the central limit theorem for $L_n$ on $K_{\alpha[n], \beta[n]}$ similar to the central limit theorem for $L_n$ on $K_n$ proved in Jensen \cite{Janson95}. Specifically, we hope to apply Theorem \ref{KurtzCLT} in the analysis. Moreover, similarly to \cite{CFIJS}, it is possible to examine the second and third order terms in $L_n$.

Finally, genetic recombination is one of the issues facing the use of coalescent processes in genetics as models of genetic drift viewed backwards in time. 
Distinct gene loci would follow different pathways of ancestry, resulting in different gene genealogies.  As a biological application, it is compelling to consider a coalescent process with multidimensional weight vectors as a means of addressing the issue of genetic recombination, and possibly, the issue of biological compatibility. 


\bibliographystyle{amsplain}

\section*{Appendix: Some alternative proofs}

\begin{proof}[Alternative proof of Lemma~\ref{lem:GelationTimeKanbn}.]
We will repeat the approach in \eqref{eqn:multStirling}.
By Stirling's approximation and equation (\ref{eqn:zetasol}), we have
\begin{align}\label{eqn:TgelStirlingsAprox}
\sum\limits_{i_1,i_2} (i_1+i_2)^2\zeta_{i_1,i_2}(t) =&\sum\limits_{i_1,i_2} (i_1+i_2)^2{i_1^{i_2-1}i_2^{i_1-1}\alpha^{i_1}\beta^{i_2} \over i_1!i_2!}e^{-(\beta i_1+ \alpha i_2)t}t^{i_1+i_2-1} \nonumber \\
=&t^{-1}\sum\limits_{i_1,i_2} {(i_1+i_2)^2 \over i_1^{3/2} i_2^{3/2}}i_1^{i_2-i_1}i_2^{i_1-i_2}
e^{-(\beta t-\ln(\alpha t) -1)i_1}e^{-(\alpha t-\ln(\beta t) -1)i_2} \nonumber \\
&\qquad\qquad\qquad\qquad \cdot {1 \over 2\pi}\big(1+O(i_1^{-1})\big)\big(1+O(i_2^{-1})\big),
\end{align}
where
$$i_1^{i_2-i_1}i_2^{i_1-i_2}=e^{-(i_1-i_2)(\ln(i_1)-\ln(i_2))} .$$
Next, we plug in $i_1=x$ and $i_2=cx$ into the exponent in \eqref{eqn:TgelStirlingsAprox}, obtaining
\begin{align}\label{eqn:expTgel}
-(i_1-i_2)\big(\ln(i_1)-\ln(i_2)\big)-&\big(\beta t-\ln(\alpha t) -1\big)i_1-\big(\alpha t-\ln(\beta t) -1\big)i_2 \nonumber \\
&=-\Big[(c\alpha+\beta)t-\ln(\alpha t)-c\ln(\beta t)+(c-1)\ln{c}-(c+1)\Big]x.
\end{align}
The maximal value of the exponent \eqref{eqn:expTgel} is therefore achieved when
\begin{equation}\label{eqn:tfromc}
t={c+1 \over c\alpha+\beta}.
\end{equation}
We plug in the optimal value \eqref{eqn:tfromc} into \eqref{eqn:expTgel} with the exponent in \eqref{eqn:TgelStirlingsAprox} becoming
equal to
\begin{equation}\label{eqn:exp_optimalT}
\left[c\ln\left({(c+1)\beta \over c(c\alpha+\beta)}\right) \,+\ln\left({(c+1)c\alpha \over c\alpha+\beta}\right)\right]x.
\end{equation}
Now, since $\ln(x)$ is a strictly concave function,
\begin{equation}\label{eqn:exp_optimalTconcave}
c\ln\left({(c+1)\beta \over c(c\alpha+\beta)}\right) \,+\ln\left({(c+1)c\alpha \over c\alpha+\beta}\right) \leq (c+1)\ln{1}=0
\end{equation}
with the equality obtained if and only if 
\begin{equation}\label{eqn:optimal_c}
c=\sqrt{\beta \over \alpha}.
\end{equation}
Thus, substituting \eqref{eqn:optimal_c} into \eqref{eqn:tfromc} yields $t={1 \over \sqrt{\alpha \beta}}$.
Indeed, the exponent in \eqref{eqn:TgelStirlingsAprox} may equal zero if and only if $t={1 \over \sqrt{\alpha \beta}}$.
If $t<{1 \over \sqrt{\alpha \beta}}$, the series  \eqref{eqn:TgelStirlingsAprox} converges.
While taking $t={1 \over \sqrt{\alpha \beta}}$, we have the portion of the series \eqref{eqn:TgelStirlingsAprox} corresponding to the indices satisfying 
$i_1\sqrt{\beta}-i_2\sqrt{\alpha}=o(i_1+i_2)$
diverging to infinity since $\sum\limits_i {1 \over i}=\infty$.

\end{proof}

Next, we give an alternative proof of Lemma \ref{lem:novoMcLeod} that uses differential equations approach in the style of \cite{McLeod62}.
\begin{proof}[Alternative proof of Lemma \ref{lem:novoMcLeod}.]
Let $x=u{\partial \over \partial u}s(u,v)$ and $y=v{\partial \over \partial v}s(u,v)$ for all $u,v \geq 0$ for which the series converge. 
The expression for the Jacobian  ${\partial(x,y) \over \partial(u,v)}$ is obtained by first taking the partial derivatives of \eqref{eqn:recSode} and arriving with
$$u{\partial^2 s \over \partial u^2} ={xy \over u(1-y)}-{1-x \over 1-y}v{\partial^2 s \over \partial u \partial v}
~~\text{ and }~~v{\partial^2 s \over \partial v^2} ={xy \over v(1-x)}-{1-y \over 1-x}u{\partial^2 s \over \partial u \partial v}.$$
Next, we substitute the above into the determinant, arriving with
\begin{align}\label{eqn:xytouvJacobian}
{\partial(x,y) \over \partial(u,v)}&={xy \over uv}+y{u \over v}{\partial^2 s \over \partial u^2}+x{v \over u}{\partial^2 s \over \partial v^2}+uv{\partial^2 s \over \partial u^2}{\partial^2 s \over \partial v^2}-uv\left({\partial^2 s \over \partial u \partial v}\right)^2 \nonumber \\
&={xy \over uv}+y{u \over v}{\partial^2 s \over \partial u^2}+x{v \over u}{\partial^2 s \over \partial v^2}+xy{v \over u(1-y)}{\partial^2 s \over \partial v^2}+xy{u \over v(1-x)}{\partial^2 s \over \partial u^2} \nonumber \\
&= {\partial s \over \partial u}{\partial s \over \partial v}\left(1+{u^2\over x(1-x)}{\partial^2 s \over \partial u^2} 
+{v^2 \over y(1-y)}{\partial^2 s \over \partial v^2} \right).
\end{align}
The above expression \eqref{eqn:xytouvJacobian} implies ${\partial(x,y) \over \partial(u,v)}>0$ for $(u,v) \in \mathbb{R}_+^2$ in a small enough neighborhood of $(0,0)$, insuring $x,y<1$.

\medskip
\noindent
Equation \eqref{eqn:recSode} rewrites as
$\,s=x+y-xy\,$
with partial derivatives ${\partial s \over \partial x}=1-y$ and ${\partial s \over \partial y}=1-x$.

\medskip
\noindent
Next, let $\tilde{u}=xe^{-y}$ and $\tilde{v}=ye^{-x}$. The Jacobian ${\partial(\tilde{u},\tilde{v}) \over \partial(x,y)}=(1-xy)e^{-x}e^{-y} =0$ if and only if $xy= 1$. 
Therefore, for $xy<1$, we have 
\begin{align*}
1-y={\partial s \over \partial x} = e^{-y} {\partial s \over \partial \tilde{u}}-ye^{-x}{\partial s \over \partial \tilde{v}}
&= {1 \over x} \tilde{u}{\partial s \over \partial \tilde{u}}-\tilde{v}{\partial s \over \partial \tilde{v}}\\
\text{ and } \qquad
1-x={\partial s \over \partial y} = -xe^{-y} {\partial s \over \partial \tilde{u}}+e^{-x}{\partial s \over \partial \tilde{v}}
&= -\tilde{u}{\partial s \over \partial \tilde{u}}+{1 \over y} \tilde{v}{\partial s \over \partial \tilde{v}},
\end{align*}
yielding
$$\tilde{u}{\partial s \over \partial \tilde{u}}=x=u{\partial s \over \partial u} \quad \text{ and } \quad \tilde{v}{\partial s \over \partial \tilde{v}}=y=v{\partial s \over \partial v}$$
with $s(u,v)\big|_{(\tilde{u},\tilde{v})=(0,0)}=s(0,0)=0$.
Hence, $(u,v)$ and $(\tilde{u},\tilde{v})$ as functions of $(x,y)$ will coincide in the whole domain $$\{(x,y) \in \mathbb{R}_+^2 \,:\, xy <1\},$$
and
$$x=\tilde{u}{\partial \over \partial u}s(\tilde{u},\tilde{v}) \quad\text{ and }\quad y=\tilde{v}{\partial \over \partial v}s(\tilde{u},\tilde{v}).$$
Here, by Prop.~\ref{prop:xyvsone}, $(x,y)$ is the smallest solution of $\tilde{u}=xe^{-y}$ and $\tilde{v}=ye^{-x}$.
Equations \eqref{eqn:xtytdsuv} follow.
\end{proof}

\end{document}